\documentclass{amsart}
\usepackage{amsfonts}
\usepackage{amsmath}
\usepackage{amssymb}
\usepackage{eurosym}
\usepackage{geometry}
\usepackage{xcolor}
\usepackage{comment}
\usepackage{pifont}

\usepackage{hyperref}
\usepackage{cite}

\newcommand{\red}{\color{red}}
\newcommand{\blue}{\color{blue}}

\theoremstyle{plain}
\newtheorem{theorem}{Theorem}[section]

\newtheorem{definition} {Definition}[section]

\newtheorem{lemma} {Lemma}[section]

\newtheorem{remark} {Remark}[section]

\begin{document}

\title{Classification of $4$-dimensional complex Poisson algebras}

\thanks{The second author is supported by the PCI of the UCA `Teor\'\i a de Lie y Teor\'\i a de Espacios de Banach' and by the PAI with project number FQM298.}

\author[H. Abdelwahab]{Hani Abdelwahab}
\address{Hani Abdelwahab.
\newline \indent Mansoura University, Faculty of Science, Department of Mathematics (Egypt).}
\email{{\tt haniamar1985@gmail.com}}



\author[J.M. Sánchez]{José María Sánchez}
\address{José María Sánchez. \newline \indent University of Cádiz, Department of Mathematics, Puerto Real (Spain).}
\email{{\tt txema.sanchez@uca.es}}

\thispagestyle{empty}

\begin{abstract}
The present paper is devoted to the complete classification of $4$-dimensional complex Poisson algebras, taking into account a classification, up to isomorphism, of the  complex commutative associative algebras of dimension $4$, as well as by using a Lie algebra classification, up to isomorphism.

\bigskip

{\it 2020MSC}: 17A30, 17B63.

 
\end{abstract}

\maketitle


\thispagestyle{empty}

\section{Introduction}

Poisson algebras are vector spaces equipped with two binary operations: a commutative product and a Lie bracket, which satisfy a certain compatibility condition known as the Leibniz rule. 
Introduced in Hamiltonian mechanics as the dual of the category of classical
mechanical systems, Poisson algebras have a fundamental role in the study of
quantum groups, differential geometry, noncommutative geometry, integrable
systems, quantum field theory or vertex operator algebras (see \cite 
{crainic, dr, gra2013, Hue, Kon, LPV, vanB}). 
Additionally, this class of algebras plays an important role in many areas of mathematics including symplectic geometry,  representation theory,
quantum field theory and
algebraic geometry. Poisson algebras can be
thought of as the algebraic counterpart of Poisson manifolds which are
smooth manifolds $M$ whose commutative algebra $C^{\infty} (M, \, {\mathbb{R}%
})$ of real smooth functions is endowed with a Lie bracket $[-, -]$
satisfying the Leibniz rule, i.e., $C^{\infty} (M, \, {\mathbb{R}})$ is a
Poisson algebra. Here, they are used to study geometric structures that preserve certain properties under deformation.

Classification algebras is important in several areas of physics and mathematics. Understanding the classification of algebras of small dimensions is often a first step to obtain the corresponding to larger algebras. Concretely, by classifying Poisson algebras we can identify their underlying symmetries and study their geometric properties, which can lead to new insights into the behavior of physical systems.

In \cite{afm} the authors developed a method to obtain the algebraic classification of Poisson algebras defined on a commutative associative algebra, and they applied it to obtain the classification of the $3$-dimensional complex Poisson algebras. In addition, they also study the algebraic classification of the Poisson algebras defined on a commutative associative null-filiform or filiform algebra.
This method  focus on  presenting a procedure to classify all the Poisson algebras associated with a given commutative associative algebra. We briefly explain the method. Pick an   arbitrary commutative associative algebra $ \mathcal{P} $. Further, compute the set $Z^2(\mathcal{P},\mathcal{P})$ of all skew symmetric bilinear maps on $\mathcal{P}$ satisfying some adequate conditions (Definition \ref{ss}).
It is proved that
for any Poisson algebra $(\mathcal{P}, \cdot, [-,-])$ there exists $\theta \in Z^2(
\mathcal{P},\mathcal{P})$ such that the algebra is isomorphic to a Poisson algebra $(\mathcal{P}, \cdot, [-,-]_{\theta})$  associated with the given commutative associative algebra $\mathcal{P}$.
Find the orbits of the automorphisms group ${\rm Aut}(\mathcal{P})$ on $Z^2(\mathcal{P},\mathcal{P})$ by a proper action \eqref{action}. It turns out that, by  choosing a representative $\theta$ from each orbit, it is obtained all Poisson algebras $(\mathcal{P},\cdot,[-,-]_{\theta})$, up to isomorphism. 
So it is obtained all the Poisson algebras associated with a given commutative associative algebra.

In the present paper, the authors produce a  complete classification of $4$-dimensional complex Poisson algebras, by applying the above results, taking into account the  classification, up to isomorphism, of the  complex commutative associative
algebras of dimension $4$ presented in \cite{Burde2013}, as well as, using the Lie algebra  classification, up to isomorphism, given in \cite{Burde99}.


\section{The algebraic classification method}
\label{sec1}

In this section, we recall the method to obtain the algebraic classification of the Poisson algebras over an arbitrary field $\mathbb{F}$ of characteristic zero present in \cite{afm}, to make this work self-contained.  Let us remind some basic definitions needed in the sequel.

\begin{definition}\rm
A \emph{Poisson algebra} is a vector space $\mathcal{P}$ endowed with two bilinear operations:
\begin{enumerate}
\item An commutative associative multiplication denoted by $-\cdot- :\mathcal{P} \times \mathcal{P} \to \mathcal{P}$;

\item A Lie algebra multiplication denoted by $[-,-] :\mathcal{P}\times   \mathcal{P}\to \mathcal{P}$.
\end{enumerate}
These two operations are compatible in the sense that they satisfy the following Leibniz identity
\begin{equation}
\label{Leibnizidentity}
[x\cdot y, z]=[x,z]\cdot y + x\cdot[y,z],
\end{equation}
for any $x,y,z\in \mathcal{P}$.
The \emph{dimension} of a Poisson algebra is its dimension as vector space.
\end{definition}
\noindent The Condition \eqref{Leibnizidentity} ensures that the Lie bracket behaves like a derivation with respect to the commutative associative multiplication.

\begin{definition}\label{def_homomorphism}\rm
Consider two Poisson algebras $(\mathcal{P}_1,\cdot_1,[\cdot,\cdot]_1)$ and $(\mathcal{P}_2,\cdot_2,[\cdot,\cdot]_2)$. A \emph{Poisson algebras  homomorphism} (or just \emph{homomorphism}, when it is clear the context), is a linear map $\phi :\mathcal{P}_1 \to \mathcal{P}_2$ preserving the two products, that is, $$\phi (x\cdot_1 y) =\phi(x) \cdot_2 \phi(y), \hspace{2cm} \phi([x,y]_1) = [\phi(x),\phi(y)]_2,$$ for all $x,y\in \mathcal{P}_1$. 
\end{definition}

Now, pick any  arbitrary commutative associative algebra, we may consider all the Poisson structures defined over this algebra. This notion is captured in the following definition.

\begin{definition}\label{ss}
Let $(\mathcal{P},\cdot)$ be a commutative associative
algebra. Define $Z^2(\mathcal{P},\mathcal{P})$ to be the set of all skew symmetric bilinear maps $\theta : \mathcal{P} \times 
\mathcal{P} \to \mathcal{P}$ such that:
\begin{eqnarray*}
&& \theta(\theta(x,y),z) + \theta (\theta(y,z),x) + \theta(\theta(z,x),y) =0, \\
&&\theta(x\cdot y,z) - \theta( x,z) \cdot y-x\cdot
\theta(y,z)=0,
\end{eqnarray*}
for all $x,y,z$ in $\mathcal{P}$. Then $Z^2(\mathcal{P},\mathcal{P})\neq \emptyset$ since $\theta=0\in Z^2(\mathcal{P},\mathcal{P})$.
\end{definition}

Observe that, for $\theta \in Z^2(\mathcal{P},\mathcal{P})$, we may define on $\mathcal{P}$ a bracket $[-,-]_{\theta}:\mathcal{P}\times \mathcal{P}\to \mathcal{P}$  by 
\begin{equation}
[x,y]_{\theta} := \theta(x,y),
\label{producttheta}
\end{equation}
for any $x,y$ in $\mathcal{P}$.

\begin{lemma}
Let $(\mathcal{P},\cdot)$ be a commutative associative
algebra and $\theta \in Z^2(\mathcal{P},\mathcal{P})$. Then $( \mathcal{P},\cdot,[-,-]_{\theta})$ is a Poisson algebra endowed with the product defined in \eqref{producttheta}.
\end{lemma}
\begin{proof}
Let $(\mathcal{P},\cdot)$ be a commutative associative algebra and let $\theta \in Z^2(\mathcal{P},\mathcal{P})$. Then $(\mathcal{P},[\cdot
,\cdot]_{\theta})$ is an anticommutative algebra. Moreover, since $\theta \in Z^2(\mathcal{P},\mathcal{P})$, we have
\begin{align*}
& [[x,y]_{\theta}   ,z]_{\theta} +[[y,z]_{\theta}   ,x]_{\theta}+
[[z,x]_{\theta}   ,y]_{\theta}=\theta(\theta(x,y),z) + \theta (\theta(y,z),x) + \theta(\theta(z,x),y) =0,\\
& [x\cdot  y,z]_{\theta} -[x,z]_{\theta} \cdot y -x\cdot [y,z]_{\theta} = \theta(x\cdot y,z) - \theta (x,z) \cdot y -x\cdot \theta(y,z)=0,
\end{align*}
for $x,y,z$ in $\mathcal{P}$, as desired.
\end{proof}

In the other way around, we may  proof that
if $(\mathcal{P},\cdot ,[-,-])$ is a Poisson algebra then there exists $\theta \in Z^2(
\mathcal{P},\mathcal{P})$ such that $(\mathcal{P}, \cdot, [-,-]_{\theta}) \cong (\mathcal{P},\cdot,[-,-])$. Indeed, let us consider the skew symmetric bilinear map $\theta : \mathcal{P}\times \mathcal{P} \to \mathcal{P}$ defined by $\theta(x,y) := [x,y]$ for $x,y \in \mathcal{P}$. Then $\theta \in Z^2(
\mathcal{P},\mathcal{P})$ and $(
\mathcal{P},\cdot ,[-,-]_{\theta}) = (\mathcal{P},\cdot,[-,-])$.

Now, let $(\mathcal{P},\cdot)$ be a commutative associative
algebra and ${\rm Aut}(\mathcal{P})$ be the automorphism group of $\mathcal{P}$. Then we can define an action of  ${\rm Aut}(\mathcal{P})$  on $Z^2(\mathcal{P},\mathcal{P})$ by
\begin{equation}
 (\theta *\phi)(x,y) := \phi^{-1}(\theta(\phi(x), \phi(y))),
\label{action}
\end{equation}
for any $\phi \in {\rm Aut}(\mathcal{P})$ and $\theta \in Z^2(\mathcal{P},\mathcal{P})$, with $x,y$ in $\mathcal{P}$.

\begin{lemma}

Let $(\mathcal{P},\cdot)$ be a commutative associative algebra
and $\theta ,\vartheta \in Z^2(\mathcal{P},\mathcal{P})$.
Then $(\mathcal{P},\cdot ,[-,-]_{\theta})$ and $(\mathcal{P},\cdot,[-,-]_{\vartheta})$ are
isomorphic if and only if there exists  $\phi \in {\rm Aut}(\mathcal{P})$ such that $\theta *\phi =\vartheta$.
\end{lemma}
\begin{proof}
If there exists  $\phi \in {\rm Aut}(\mathcal{P})$ such that  $\theta *\phi =\vartheta$ then $\phi : (\mathcal{P},\cdot,[-,-]_{\vartheta}) \to $ $(\mathcal{P},\cdot,[-,-]_{\theta})$ is an isomorphism since $\phi(\vartheta(x,y) =\theta(\phi(x),\phi(y))$, with $x,y$ in $\mathcal{P}$. On the other hand, if $\phi : (\mathcal{P},\cdot,[-,-]_{\vartheta}) \to $ $(\mathcal{P},\cdot,[-,-]_{\theta})$ is an isomorphism of Poisson algebras, then $\phi \in {\rm Aut}(\mathcal{P})$ and $\phi([x,y]_{\vartheta}) = [\phi(x),\phi(y)]_{\theta}$, with $x,y$ in $\mathcal{P}$. Hence $\vartheta(x,y) =\phi^{-1}(\theta(\phi(x),\phi(y)))= (\theta * \phi)(x,y)$ and therefore $\theta *\phi = \vartheta$.
\end{proof}

Hence, we have a procedure to classify all the Poisson algebras associated to a given commutative associative algebra $(\mathcal{P},\cdot)$. It is performed in three steps:

\begin{enumerate}
\item Compute $Z^2(\mathcal{P},\mathcal{P})$.
\item Find the orbits of ${\rm Aut}(\mathcal{P})$ on $Z^2(\mathcal{P},\mathcal{P})$.
\item Choose a representative $\theta$ from each orbit and then construct the Poisson algebra $(\mathcal{P},\cdot,[-,-]_{\theta})$.
\end{enumerate}

\begin{remark}
Similarly, we can construct an analogous method for classifying the $4$-dimensional Poisson algebras from the
classification of Lie algebras of dimension four. 
\end{remark}

Let us denote the following notation. Let $\{e_1,e_2,\dots, e_n\}$ be a fixed basis of a commutative associative algebra $(\mathcal{P},\cdot)$. We define $\Lambda^2(\mathcal{P},\mathbb{F}) := {\rm span}_{\mathbb F}\{\Delta_{i,j} : 1\leq i<j\leq n\}$, where each $\Delta _{i,j}:\mathcal{P}\times \mathcal{P} \to \mathbb{F}$ is the skew-symmetric bilinear form defined by
\[
\Delta_{i,j}(e_l,e_m) :=\left\{ 
\begin{tabular}{ll}
$1,$ & if $(i,j) =(l,m),$ \\ 
$-1,$ & if $(i,j) =(m,l), $ \\ 
$0,$ & otherwise.
\end{tabular}
\right. 
\]
Now, if $\theta \in Z^2(\mathcal{P},\mathcal{P})$, then $\theta$ can be uniquely written as $\theta(x,y) = \sum_{i=1}^n B_i(x,y)e_i$ where $B_1,B_2,\dots,B_n$ is a sequence of skew symmetric bilinear forms on $\mathcal{P}$. Also, we may write $\theta = (B_1,B_2,\dots, B_n)$ . Let $\phi^{-1} \in {\rm Aut}(\mathcal{P})$ be given by the matrix $(b_{ij})$. If $(\theta *\phi)(x,y) = \sum_{i=1}^n B_i'(x,y)e_i$ then $B_i' = \sum_{j=1}^n b_{ij} \phi^t B_j\phi$.

\begin{remark}
Note that if $(\mathcal{P}_1,\cdot_1,[\cdot,\cdot]_1)$ and $(\mathcal{P}_2,\cdot_2,[\cdot,\cdot]_2)$ are two isomorphic Poisson algebras, then the commutative associative algebras $(\mathcal{P}_1,\cdot_1)$ and $(\mathcal{P}_2,\cdot_2)$ are isomorphic. So, given two non-isomorphic commutative associative algebras $\mathcal{P} $ and $\mathcal{P'}$, we have that any Poisson structure on $\mathcal{P}$ is
not isomorphic to any Poisson structure on $\mathcal{P'}$.
\end{remark}
 
\begin{remark}
\label{[1 0]}Let $X=
\begin{pmatrix} \alpha & \beta \end{pmatrix} \in {\mathcal{M}}_{1\times 2}({\mathbb{F}})$ and $X\neq 0$. Then there exists an invertible matrix $A\in {\mathcal{M}}_{2\times 2}({\mathbb{F}})$
such that $XA= \begin{pmatrix} 1 & 0 \end{pmatrix}$. In fact,   first assume that $\alpha \neq 0$. Then $\begin{pmatrix} \alpha & \beta
\end{pmatrix}
\begin{pmatrix} \alpha ^{-1} & -\beta \\ 0 & \alpha \end{pmatrix} = \begin{pmatrix} 1 & 0 \end{pmatrix}$.  Now, if  $\alpha =0$  then $\begin{pmatrix} 0 & \beta \end{pmatrix}
\begin{pmatrix} 0 & 1 \\ \beta^{-1} & 0 \end{pmatrix} =\allowbreak 
\begin{pmatrix} 1 & 0 \end{pmatrix}$.
\end{remark}

\section{Poisson algebras of dimension $4$}
\label{3y4}

From now, we present the classification, up to isomorphism, of the Poisson algebras for dimension $4$ over the field of complex numbers $\mathbb{C}$. 
For simplicity, every time we write the multiplication table of a Poisson algebra the products of basic elements whose values are zero or can be recovered by the commutativity, in the case of $-\cdot-$, or by the anticommutativity, in the case of $[-,-]$, are omitted.
First we recall the  classification, up to isomorphism, of the  complex commutative associative
algebras of dimension $4$ presented in \cite{Burde2013}:

\begin{theorem}
\cite{Burde2013} Let $\mathcal{A}$ be a complex commutative associative
algebra of dimension $4$. Then $\mathcal{A}$ is isomorphic to one of the
following algebras:

\begin{itemize}
\item $\mathcal{A}_{01}:$ trivial algebra.

\item $\mathcal{A}_{02}:e_{1}^{2}=e_{2}.$

\item $\mathcal{A}_{03}:e_{1}^{2}=e_{3},e_{2}^{2}=e_{3}.$

\item $\mathcal{A}_{04}:e_{1}^{2}=e_{2},e_{1}\cdot e_{2}=e_{3}.$

\item $\mathcal{A}_{05}:e_{1}^{2}=-e_{3},e_{1}\cdot
e_{2}=e_{4},e_{2}^{2}=e_{3}.$

\item $\mathcal{A}_{06}:e_{1}\cdot e_{2}=e_{4},e_{2}^{2}=e_{3}.$

\item $\mathcal{A}_{07}:e_{1}^{2}=e_{4},e_{2}^{2}=e_{4},e_{3}^{2}=e_{4}.$

\item $\mathcal{A}_{08}:e_{1}^{2}=e_{2},e_{1}\cdot
e_{2}=e_{4},e_{3}^{2}=e_{4}.$

\item $\mathcal{A}_{09}:e_{1}^{2}=e_{2},e_{1}\cdot e_{2}=e_{3},e_{1}\cdot
e_{3}=e_{4},e_{2}^{2}=e_{4}.$

\item $\mathcal{A}%
_{10}:e_{1}^{2}=e_{1},e_{2}^{2}=e_{2},e_{3}^{2}=e_{3},e_{4}^{2}=e_{4}.$

\item $\mathcal{A}%
_{11}:e_{1}^{2}=e_{1},e_{2}^{2}=e_{2},e_{3}^{2}=e_{3},e_{3}\cdot
e_{4}=e_{4}. $

\item $\mathcal{A}_{12}:e_{1}^{2}=e_{1},e_{1}\cdot
e_{2}=e_{2},e_{3}^{2}=e_{3},e_{3}\cdot e_{4}=e_{4}.$

\item $\mathcal{A}_{13}:e_{1}^{2}=e_{1},e_{2}^{2}=e_{2},e_{2}\cdot
e_{3}=e_{3},e_{2}\cdot e_{4}=e_{4}.$

\item $\mathcal{A}_{14}:e_{1}^{2}=e_{1},e_{2}^{2}=e_{2},e_{2}\cdot
e_{3}=e_{3},e_{2}\cdot e_{4}=e_{4},e_{3}^{2}=e_{4}.$

\item $\mathcal{A}_{15}:e_{1}^{2}=e_{1},e_{1}\cdot e_{2}=e_{2},e_{1}\cdot
e_{3}=e_{3},e_{1}\cdot e_{4}=e_{4}.$

\item $\mathcal{A}_{16}:e_{1}^{2}=e_{1},e_{1}\cdot e_{2}=e_{2},e_{1}\cdot
e_{3}=e_{3},e_{1}\cdot e_{4}=e_{4},e_{2}^{2}=e_{3}.$

\item $\mathcal{A}_{17}:e_{1}^{2}=e_{1},e_{1}\cdot e_{2}=e_{2},e_{1}\cdot
e_{3}=e_{3},e_{1}\cdot e_{4}=e_{4},e_{2}^{2}=e_{3},e_{2}\cdot e_{3}=e_{4}.$

\item $\mathcal{A}_{18}:e_{1}^{2}=e_{1},e_{1}\cdot e_{2}=e_{2},e_{1}\cdot
e_{3}=e_{3},e_{1}\cdot e_{4}=e_{4},e_{2}^{2}=e_{4},e_{3}^{2}=e_{4}.$

\item $\mathcal{A}_{19}:e_{1}^{2}=e_{1},e_{2}^{2}=e_{2},e_{3}^{2}=e_{3}.$

\item $\mathcal{A}_{20}:e_{1}^{2}=e_{1},e_{2}^{2}=e_{2},e_{2}\cdot
e_{3}=e_{3}.$

\item $\mathcal{A}_{21}:e_{1}^{2}=e_{1},e_{1}\cdot e_{2}=e_{2},e_{1}\cdot
e_{3}=e_{3}.$

\item $\mathcal{A}_{22}:e_{1}^{2}=e_{1},e_{1}\cdot e_{2}=e_{2},e_{1}\cdot
e_{3}=e_{3},e_{2}^{2}=e_{3}.$

\item $\mathcal{A}_{23}:e_{1}^{2}=e_{1},e_{2}^{2}=e_{2}.$

\item $\mathcal{A}_{24}:e_{1}^{2}=e_{1},e_{2}^{2}=e_{2},e_{3}^{2}=e_{4}.$

\item $\mathcal{A}_{25}:e_{1}^{2}=e_{1},e_{1}\cdot e_{2}=e_{2}.$

\item $\mathcal{A}_{26}:e_{1}^{2}=e_{1},e_{1}\cdot
e_{2}=e_{2},e_{3}^{2}=e_{4}.$

\item $\mathcal{A}_{27}:e_{1}^{2}=e_{1}.$

\item $\mathcal{A}_{28}:e_{1}^{2}=e_{1},e_{2}\cdot e_{3}=e_{4}.$

\item $\mathcal{A}_{29}:e_{1}^{2}=e_{1},e_{2}^{2}=e_{3}.$

\item $\mathcal{A}_{30}:e_{1}^{2}=e_{1},e_{2}^{2}=e_{3},e_{2}\cdot
e_{3}=e_{4}.$
\end{itemize}
\end{theorem}

Now, we present the main result of the paper. Either

\begin{theorem}
\label{main}
Let $\mathcal{P}$ be a complex Poisson algebra of dimension $4$. Then $%
\mathcal{P}$ is isomorphic to one of the Lie algebras given in \cite[Lemma 3]%
{Burde99} or to one of the following algebras:

\begin{itemize}
\item $\mathcal{P}_{4,1}:e_{1}^{2}=e_{2}.$

\item $\mathcal{P}_{4,2}:\left\{ 
\begin{tabular}{l}
$e_{1}^{2}=e_{2},$ \\ 
$\left[ e_{1},e_{3}\right] =e_{2}.$%
\end{tabular}%
\right. $

\item $\mathcal{P}_{4,3}:\left\{ 
\begin{tabular}{l}
$e_{1}^{2}=e_{2},$ \\ 
$\left[ e_{3},e_{4}\right] =e_{2}.$%
\end{tabular}%
\right. $

\item $\mathcal{P}_{4,4}^{\alpha }:\left\{ 
\begin{tabular}{l}
$e_{1}^{2}=e_{2},$ \\ 
$\left[ e_{1},e_{3}\right] =e_{3},\left[ e_{1},e_{4}\right] =\alpha e_{4}.$%
\end{tabular}%
\right. $

\item $\mathcal{P}_{4,5}:\left\{ 
\begin{tabular}{l}
$e_{1}^{2}=e_{2},$ \\ 
$\left[ e_{1},e_{4}\right] =e_{2},\left[ e_{1},e_{3}\right] =e_{3}.$%
\end{tabular}%
\right. $

\item $\mathcal{P}_{4,6}:\left\{ 
\begin{tabular}{l}
$e_{1}^{2}=e_{2},$ \\ 
$\left[ e_{3},e_{4}\right] =e_{2},\left[ e_{1},e_{3}\right] =e_{3},\left[
e_{1},e_{4}\right] =-e_{4}.$%
\end{tabular}%
\right. $

\item $\mathcal{P}_{4,7}:\left\{ 
\begin{tabular}{l}
$e_{1}^{2}=e_{2},$ \\ 
$\left[ e_{1},e_{3}\right] =e_{3},\left[ e_{1},e_{4}\right] =e_{3}+e_{4}.$%
\end{tabular}%
\right. $

\item $\mathcal{P}_{4,8}:\left\{ 
\begin{tabular}{l}
$e_{1}^{2}=e_{2},$ \\ 
$\left[ e_{1},e_{4}\right] =e_{3}.$%
\end{tabular}%
\right. $

\item $\mathcal{P}_{4,9}:\left\{ 
\begin{tabular}{l}
$e_{1}^{2}=e_{2},$ \\ 
$\left[ e_{1},e_{3}\right] =e_{2},\left[ e_{1},e_{4}\right] =e_{3}.$%
\end{tabular}%
\right. $

\item $\mathcal{P}_{4,10}:\left\{ 
\begin{tabular}{l}
$e_{1}^{2}=e_{2},$ \\ 
$\left[ e_{3},e_{4}\right] =e_{2},\left[ e_{1},e_{4}\right] =e_{3}.$%
\end{tabular}%
\right. $

\item $\mathcal{P}_{4,11}:\left\{ 
\begin{tabular}{l}
$e_{1}^{2}=e_{2},$ \\ 
$\left[ e_{3},e_{4}\right] =e_{3}.$%
\end{tabular}%
\right. $

\item $\mathcal{P}_{4,12}:\left\{ 
\begin{tabular}{l}
$e_{1}^{2}=e_{2},$ \\ 
$\left[ e_{1},e_{4}\right] =e_{2},\left[ e_{3},e_{4}\right] =e_{3}.$%
\end{tabular}%
\right. $

\item $\mathcal{P}_{4,13}^{\alpha }:\left\{ 
\begin{tabular}{l}
$e_{1}^{2}=e_{3},e_{2}^{2}=e_{3},$ \\ 
$\left[ e_{1},e_{2}\right] =\alpha e_{3},\left[ e_{1},e_{4}\right] =e_{4}.$%
\end{tabular}%
\right. $

\item $\mathcal{P}_{4,14}:\left\{ 
\begin{tabular}{l}
$e_{1}^{2}=e_{3},e_{2}^{2}=e_{3},$ \\ 
$\left[ e_{1},e_{4}\right] =e_{3}.$%
\end{tabular}%
\right. $

\item $\mathcal{P}_{4,15}:\left\{ 
\begin{tabular}{l}
$e_{1}^{2}=e_{3},e_{2}^{2}=e_{3},$ \\ 
$\left[ e_{1},e_{4}\right] =e_{3},\left[ e_{1},e_{2}\right] =e_{4}.$%
\end{tabular}%
\right. $

\item $\mathcal{P}_{4,16}^{\alpha }:\left\{ 
\begin{tabular}{l}
$e_{1}^{2}=e_{3},e_{2}^{2}=e_{3},$ \\ 
$\left[ e_{1},e_{2}\right] =\alpha e_{3}.$%
\end{tabular}%
\right. $

\item $\mathcal{P}_{4,17}:\left\{ 
\begin{tabular}{l}
$e_{1}^{2}=e_{3},e_{2}^{2}=e_{3},$ \\ 
$\left[ e_{1},e_{2}\right] =e_{4}.$%
\end{tabular}%
\right. $

\item $\mathcal{P}_{4,18}:\left\{ 
\begin{tabular}{l}
$e_{1}^{2}=e_{3},e_{2}^{2}=e_{3},$ \\ 
$\left[ e_{1},e_{4}\right] =e_{3},\left[ e_{2},e_{4}\right] =ie_{3}.$%
\end{tabular}%
\right. $

\item $\mathcal{P}_{4,19}:\left\{ 
\begin{tabular}{l}
$e_{1}^{2}=e_{3},e_{2}^{2}=e_{3},$ \\ 
$\left[ e_{1},e_{4}\right] =e_{3},\left[ e_{2},e_{4}\right] =ie_{3},\left[
e_{1},e_{2}\right] =e_{4}.$%
\end{tabular}%
\right. $

\item $\mathcal{P}_{4,20}^{\alpha }:\left\{ 
\begin{tabular}{l}
$e_{1}^{2}=e_{3},e_{2}^{2}=e_{3},$ \\ 
$\left[ e_{1},e_{2}\right] =\alpha e_{3},\left[ e_{1},e_{4}\right] =e_{4},%
\left[ e_{2},e_{4}\right] =ie_{4}.$%
\end{tabular}%
\right. $

\item $\mathcal{P}_{4,21}:e_{1}^{2}=e_{2},e_{1}\cdot e_{2}=e_{3}.$

\item $\mathcal{P}_{4,22}:\left\{ 
\begin{tabular}{l}
$e_{1}^{2}=e_{2},e_{1}\cdot e_{2}=e_{3},$ \\ 
$\left[ e_{1},e_{4}\right] =e_{4}.$%
\end{tabular}%
\right. $

\item $\mathcal{P}_{4,23}:\left\{ 
\begin{tabular}{l}
$e_{1}^{2}=e_{2},e_{1}\cdot e_{2}=e_{3},$ \\ 
$\left[ e_{1},e_{4}\right] =e_{3}.$%
\end{tabular}%
\right. $

\item $\mathcal{P}_{4,24}:e_{1}^{2}=e_{3},e_{2}^{2}=e_{4}.$

\item $\mathcal{P}_{4,25}^{\alpha }:\left\{ 
\begin{tabular}{l}
$e_{1}^{2}=e_{3},e_{2}^{2}=e_{4},$ \\ 
$\left[ e_{1},e_{2}\right] =e_{3}+\alpha e_{4}.$%
\end{tabular}%
\right. $

\item $\mathcal{P}_{4,26}:\left\{ 
\begin{tabular}{l}
$e_{1}\cdot e_{2}=e_{4},e_{2}^{2}=e_{3},$ \\ 
$\left[ e_{1},e_{2}\right] =e_{1},\left[ e_{1},e_{3}\right] =2e_{4},\left[
e_{2},e_{4}\right] =-e_{4}.$%
\end{tabular}%
\right. $

\item $\mathcal{P}_{4,27}:\left\{ 
\begin{tabular}{l}
$e_{1}\cdot e_{2}=e_{4},e_{2}^{2}=e_{3},$ \\ 
$\left[ e_{1},e_{2}\right] =e_{3}.$%
\end{tabular}%
\right. $

\item $\mathcal{P}_{4,28}^{\alpha }:\left\{ 
\begin{tabular}{l}
$e_{1}\cdot e_{2}=e_{4},e_{2}^{2}=e_{3},$ \\ 
$\left[ e_{1},e_{2}\right] =\alpha e_{4}.$%
\end{tabular}%
\right. $

\item $\mathcal{P}_{4,29}:\left\{ 
\begin{tabular}{l}
$e_{1}\cdot e_{2}=e_{4},e_{3}^{2}=e_{4},$ \\ 
$\left[ e_{1},e_{3}\right] =-e_{1},\left[ e_{2},e_{3}\right] =e_{2},\left[
e_{1},e_{2}\right] =e_{3}.$%
\end{tabular}%
\right. $

\item $\mathcal{P}_{4,30}:\left\{ 
\begin{tabular}{l}
$e_{1}\cdot e_{2}=e_{4},e_{3}^{2}=e_{4},$ \\ 
$\left[ e_{1},e_{3}\right] =e_{4}.$%
\end{tabular}%
\right. $

\item $\mathcal{P}_{4,31}^{\alpha }:\left\{ 
\begin{tabular}{l}
$e_{1}\cdot e_{2}=e_{4},e_{3}^{2}=e_{4},$ \\ 
$\left[ e_{1},e_{2}\right] =\alpha e_{4}.$%
\end{tabular}%
\right. $

\item $\mathcal{P}_{4,32}:e_{1}^{2}=e_{2},e_{1}\cdot
e_{2}=e_{4},e_{3}^{2}=e_{4}.$

\item $\mathcal{P}_{4,33}:\left\{ 
\begin{tabular}{l}
$e_{1}^{2}=e_{2},e_{1}\cdot e_{2}=e_{4},e_{3}^{2}=e_{4},$ \\ 
$\left[ e_{1},e_{3}\right] =e_{4}.$%
\end{tabular}%
\right. $

\item $\mathcal{P}_{4,34}:e_{1}^{2}=e_{2},e_{1}\cdot e_{2}=e_{3},e_{1}\cdot
e_{3}=e_{4},e_{2}^{2}=e_{4}.$

\item $\mathcal{P}%
_{4,35}:e_{1}^{2}=e_{1},e_{2}^{2}=e_{2},e_{3}^{2}=e_{3},e_{4}^{2}=e_{4}.$

\item $\mathcal{P}%
_{4,36}:e_{1}^{2}=e_{1},e_{2}^{2}=e_{2},e_{3}^{2}=e_{3},e_{3}\cdot
e_{4}=e_{4}.$

\item $\mathcal{P}_{4,37}:e_{1}^{2}=e_{1},e_{1}\cdot
e_{2}=e_{2},e_{3}^{2}=e_{3},e_{3}\cdot e_{4}=e_{4}.$

\item $\mathcal{P}_{4,38}:e_{1}^{2}=e_{1},e_{2}^{2}=e_{2},e_{2}\cdot
e_{3}=e_{3},e_{2}\cdot e_{4}=e_{4}.$

\item $\mathcal{P}_{4,39}:\left\{ 
\begin{tabular}{l}
$e_{1}^{2}=e_{1},e_{2}^{2}=e_{2},e_{2}\cdot e_{3}=e_{3},e_{2}\cdot
e_{4}=e_{4},$ \\ 
$\left[ e_{3},e_{4}\right] =e_{3}.$%
\end{tabular}%
\right. $

\item $\mathcal{P}_{4,40}:e_{1}^{2}=e_{1},e_{2}^{2}=e_{2},e_{2}\cdot
e_{3}=e_{3},e_{2}\cdot e_{4}=e_{4},e_{3}^{2}=e_{4}.$

\item $\mathcal{P}_{4,41}^{\alpha }:\left\{ 
\begin{tabular}{l}
$e_{1}^{2}=e_{1},e_{1}\cdot e_{2}=e_{2},e_{1}\cdot e_{3}=e_{3},e_{1}\cdot
e_{4}=e_{4},$ \\ 
$\left[ e_{2},e_{4}\right] =\alpha e_{2},\left[ e_{3},e_{4}\right] =e_{3}.$%
\end{tabular}%
\right. $

\item $\mathcal{P}_{4,42}:\left\{ 
\begin{tabular}{l}
$e_{1}^{2}=e_{1},e_{1}\cdot e_{2}=e_{2},e_{1}\cdot e_{3}=e_{3},e_{1}\cdot
e_{4}=e_{4},$ \\ 
$\left[ e_{2},e_{4}\right] =e_{2},\left[ e_{3},e_{4}\right] =e_{2}+e_{3}.$%
\end{tabular}%
\right. $

\item $\mathcal{P}_{4,43}:e_{1}^{2}=e_{1},e_{1}\cdot e_{2}=e_{2},e_{1}\cdot
e_{3}=e_{3},e_{1}\cdot e_{4}=e_{4}.$

\item $\mathcal{P}_{4,44}:\left\{ 
\begin{tabular}{l}
$e_{1}^{2}=e_{1},e_{1}\cdot e_{2}=e_{2},e_{1}\cdot e_{3}=e_{3},e_{1}\cdot
e_{4}=e_{4},$ \\ 
$\left[ e_{3},e_{4}\right] =e_{2}.$%
\end{tabular}%
\right. $

\item $\mathcal{P}_{4,45}:\left\{ 
\begin{tabular}{l}
$e_{1}^{2}=e_{1},e_{1}\cdot e_{2}=e_{2},e_{1}\cdot e_{3}=e_{3},e_{1}\cdot
e_{4}=e_{4},$ \\ 
$\left[ e_{3},e_{4}\right] =e_{2},\left[ e_{2},e_{3}\right] =e_{3},\left[
e_{2},e_{4}\right] =-e_{4}.$%
\end{tabular}%
\right. $

\item $\mathcal{P}_{4,46}:e_{1}^{2}=e_{1},e_{1}\cdot e_{2}=e_{2},e_{1}\cdot
e_{3}=e_{3},e_{1}\cdot e_{4}=e_{4},e_{2}^{2}=e_{3}.$

\item $\mathcal{P}_{4,47}:\left\{ 
\begin{tabular}{l}
$e_{1}^{2}=e_{1},e_{1}\cdot e_{2}=e_{2},e_{1}\cdot e_{3}=e_{3},e_{1}\cdot
e_{4}=e_{4},e_{2}^{2}=e_{3},$ \\ 
$\left[ e_{2},e_{4}\right] =e_{4}.$%
\end{tabular}%
\right. $

\item $\mathcal{P}_{4,48}:\left\{ 
\begin{tabular}{l}
$e_{1}^{2}=e_{1},e_{1}\cdot e_{2}=e_{2},e_{1}\cdot e_{3}=e_{3},e_{1}\cdot
e_{4}=e_{4},e_{2}^{2}=e_{3},$ \\ 
$\left[ e_{2},e_{4}\right] =e_{3}.$%
\end{tabular}%
\right. $

\item $\mathcal{P}_{4,49}:e_{1}^{2}=e_{1},e_{1}\cdot e_{2}=e_{2},e_{1}\cdot
e_{3}=e_{3},e_{1}\cdot e_{4}=e_{4},e_{2}^{2}=e_{3},e_{2}\cdot e_{3}=e_{4}.$

\item $\mathcal{P}_{4,50}^{\alpha }:\left\{ 
\begin{tabular}{l}
$e_{1}^{2}=e_{1},e_{1}\cdot e_{2}=e_{2},e_{1}\cdot e_{3}=e_{3},e_{1}\cdot
e_{4}=e_{4},e_{2}^{2}=e_{4},e_{3}^{2}=e_{4},$ \\ 
$\left[ e_{2},e_{3}\right] =\alpha e_{4}.$%
\end{tabular}%
\right. $

\item $\mathcal{P}_{4,51}:e_{1}^{2}=e_{1},e_{2}^{2}=e_{2},e_{3}^{2}=e_{3}.$

\item $\mathcal{P}_{4,52}:e_{1}^{2}=e_{1},e_{2}^{2}=e_{2},e_{2}\cdot
e_{3}=e_{3}.$

\item $\mathcal{P}_{4,53}:e_{1}^{2}=e_{1},e_{1}\cdot e_{2}=e_{2},e_{1}\cdot
e_{3}=e_{3}.$

\item $\mathcal{P}_{4,54}:\left\{ 
\begin{tabular}{l}
$e_{1}^{2}=e_{1},e_{1}\cdot e_{2}=e_{2},e_{1}\cdot e_{3}=e_{3},$ \\ 
$\left[ e_{2},e_{3}\right] =e_{2}.$%
\end{tabular}%
\right. $

\item $\mathcal{P}_{4,55}:e_{1}^{2}=e_{1},e_{1}\cdot e_{2}=e_{2},e_{1}\cdot
e_{3}=e_{3},e_{2}^{2}=e_{3}.$

\item $\mathcal{P}_{4,56}:e_{1}^{2}=e_{1},e_{2}^{2}=e_{2}.$

\item $\mathcal{P}_{4,57}:\left\{ 
\begin{tabular}{l}
$e_{1}^{2}=e_{1},e_{2}^{2}=e_{2},$ \\ 
$\left[ e_{3},e_{4}\right] =e_{3}.$%
\end{tabular}%
\right. $

\item $\mathcal{P}_{4,58}:e_{1}^{2}=e_{1},e_{2}^{2}=e_{2},e_{3}^{2}=e_{4}.$

\item $\mathcal{P}_{4,59}:e_{1}^{2}=e_{1},e_{1}\cdot e_{2}=e_{2}.$

\item $\mathcal{P}_{4,60}:\left\{ 
\begin{tabular}{l}
$e_{1}^{2}=e_{1},e_{1}\cdot e_{2}=e_{2},$ \\ 
$\left[ e_{3},e_{4}\right] =e_{3}.$%
\end{tabular}%
\right. $

\item $\mathcal{P}_{4,61}:e_{1}^{2}=e_{1},e_{1}\cdot
e_{2}=e_{2},e_{3}^{2}=e_{4}.$

\item $\mathcal{P}_{4,62}^{\alpha }:\left\{ 
\begin{tabular}{l}
$e_{1}^{2}=e_{1},$ \\ 
$\left[ e_{2},e_{4}\right] =\alpha e_{2},\left[ e_{3},e_{4}\right] =e_{3}.$%
\end{tabular}%
\right. $

\item $\mathcal{P}_{4,63}:\left\{ 
\begin{tabular}{l}
$e_{1}^{2}=e_{1},$ \\ 
$\left[ e_{2},e_{4}\right] =e_{2},\left[ e_{3},e_{4}\right] =e_{2}+e_{3}.$%
\end{tabular}%
\right. $

\item $\mathcal{P}_{4,64}:e_{1}^{2}=e_{1}.$

\item $\mathcal{P}_{4,65}:\left\{ 
\begin{tabular}{l}
$e_{1}^{2}=e_{1},$ \\ 
$\left[ e_{3},e_{4}\right] =e_{2}.$%
\end{tabular}%
\right. $

\item $\mathcal{P}_{4,66}:\left\{ 
\begin{tabular}{l}
$e_{1}^{2}=e_{1},$ \\ 
$\left[ e_{3},e_{4}\right] =e_{2},\left[ e_{2},e_{3}\right] =e_{3},\left[
e_{2},e_{4}\right] =-e_{4}.$%
\end{tabular}%
\right. $

\item $\mathcal{P}_{4,67}^{\alpha }:\left\{ 
\begin{tabular}{l}
$e_{1}^{2}=e_{1},e_{2}\cdot e_{3}=e_{4},$ \\ 
$\left[ e_{2},e_{3}\right] =\alpha e_{4}.$%
\end{tabular}%
\right. $

\item $\mathcal{P}_{4,68}:\left\{ 
\begin{tabular}{l}
$e_{1}^{2}=e_{1},e_{2}^{2}=e_{3},$ \\ 
$\left[ e_{2},e_{4}\right] =e_{4}.$%
\end{tabular}%
\right. $

\item $\mathcal{P}_{4,69}:\left\{ 
\begin{tabular}{l}
$e_{1}^{2}=e_{1},e_{2}^{2}=e_{3},$ \\ 
$\left[ e_{2},e_{4}\right] =e_{3}.$%
\end{tabular}%
\right. $

\item $\mathcal{P}_{4,70}:e_{1}^{2}=e_{1},e_{2}^{2}=e_{3}.$

\item $\mathcal{P}_{4,71}:e_{1}^{2}=e_{1},e_{2}^{2}=e_{3},e_{2}\cdot
e_{3}=e_{4}.$
\end{itemize}
Between these algebras there are precisely the following isomorphisms:

\begin{itemize}
\item $\mathcal{P}_{4,4}^{\alpha } \cong \mathcal{P}_{4,4}^{\beta}$ if and
only if $(\alpha -\beta)(\alpha \beta -1)=0$.

\item $\mathcal{P}_{4,13}^{\alpha } \cong \mathcal{P}_{4,13}^{\beta}$ if and
only if $\alpha^2 =\beta^2$.

\item $\mathcal{P}_{4,16}^{\alpha } \cong \mathcal{P}_{4,16}^{\beta}$ if and
only if $\alpha^2 =\beta^2$.

\item $\mathcal{P}_{4,20}^{\alpha } \cong \mathcal{P}_{4,20}^{\beta}$ if and
only if $\alpha^2 =\beta^2$.

\item $\mathcal{P}_{4,25}^{\alpha } \cong \mathcal{P}_{4,25}^{\beta}$ if and
only if $\alpha =\beta$.

\item $\mathcal{P}_{4,28}^{\alpha } \cong \mathcal{P}_{4,28}^{\beta}$ if and
only if $\alpha =\beta$.

\item $\mathcal{P}_{4,31}^{\alpha } \cong \mathcal{P}_{4,31}^{\beta}$ if and
only if $\alpha^2 =\beta^2$.

\item $\mathcal{P}_{4,41}^{\alpha } \cong \mathcal{P}_{4,41}^{\beta}$ if and
only if $(\alpha -\beta)(\alpha \beta -1)=0$.

\item $\mathcal{P}_{4,50}^{\alpha } \cong \mathcal{P}_{4,50}^{\beta}$ if and
only if $\alpha^2 =\beta^2$.

\item $\mathcal{P}_{4,62}^{\alpha } \cong \mathcal{P}_{4,62}^{\beta}$ if and
only if $(\alpha -\beta)(\alpha \beta -1)=0$.

\item $\mathcal{P}_{4,67}^{\alpha } \cong \mathcal{P}_{4,67}^{\beta}$ if and
only if $\alpha^2 =\beta^2$.

\end{itemize}
\end{theorem}

\section{The proof of Theorem \ref{main}}

\underline{$\left( \mathcal{P},\cdot \right) =\mathcal{A}_{01}$.} Then $%
\mathcal{P}$ is a Lie algebra. So $\mathcal{P}$ is isomorphic to one of the
Lie algebras given in \cite[Lemma 3]{Burde99}.

\underline{$\left( \mathcal{P},\cdot \right) =\mathcal{A}_{02}$.} The
automorphism group of $\mathcal{A}_{02}$, $\text{Aut}\left( \mathcal{A}%
_{02}\right) $, consists of the automorphisms $\phi $ given by a matrix of
the following form:%
\begin{equation*}
\phi =%
\begin{pmatrix}
a_{11} & 0 & 0 & 0 \\ 
a_{21} & a_{11}^{2} & a_{23} & a_{24} \\ 
a_{31} & 0 & a_{33} & a_{34} \\ 
a_{41} & 0 & a_{43} & a_{44}%
\end{pmatrix}%
.
\end{equation*}%
Let $\theta =\left( B_{1},B_{2},B_{3},B_{4}\right) $ be an arbitrary element
of $Z^{2}\left( \mathcal{P},\mathcal{P}\right) $. Then%
\begin{eqnarray*}
B_{1} &=&0, \\
B_{2} &=&\alpha _{1}\Delta _{1,3}+\alpha _{2}\Delta _{1,4}+\alpha _{3}\Delta
_{3,4}, \\
B_{3} &=&\alpha _{4}\Delta _{1,3}+\alpha _{5}\Delta _{1,4}+\alpha _{6}\Delta
_{3,4}, \\
B_{4} &=&\alpha _{7}\Delta _{1,3}+\alpha _{8}\Delta _{1,4}+\alpha _{9}\Delta
_{3,4},
\end{eqnarray*}%
such that%
\begin{eqnarray*}
\alpha _{4}\alpha _{9}-\alpha _{6}\alpha _{7} &=&0, \\
\alpha _{5}\alpha _{9}-\alpha _{6}\alpha _{8} &=&0, \\
\alpha _{1}\alpha _{6}+\alpha _{2}\alpha _{9}-\alpha _{3}\alpha _{4}-\alpha
_{3}\alpha _{8} &=&0,
\end{eqnarray*}%
for some $\alpha _{1},\ldots ,\alpha _{9}\in \mathbb{C}$. Let $\phi =\bigl(%
a_{ij}\bigr)\in $ $\text{Aut}\left( \mathcal{A}_{02}\right) $. Write%
\begin{equation*}
\theta \ast \phi =\left( 0,\beta _{1}\Delta _{1,3}+\beta _{2}\Delta
_{1,4}+\beta _{3}\Delta _{3,4},\beta _{4}\Delta _{1,3}+\beta _{5}\Delta
_{1,4}+\beta _{6}\Delta _{3,4},\beta _{7}\Delta _{1,3}+\beta _{8}\Delta
_{1,4}+\beta _{9}\Delta _{3,4}\right) .
\end{equation*}%
Then 
\begin{eqnarray*}
\beta _{6} &=&\alpha _{6}a_{44}-\alpha _{9}a_{34}, \\
\beta _{9} &=&\alpha _{9}a_{33}-\alpha _{6}a_{43}.
\end{eqnarray*}%
By Remark \ref{[1 0]}, we may assume $\left( \alpha _{6},\alpha _{9}\right)
\in \left\{ \left( 0,0\right) ,\left( 1,0\right) \right\} $. Assume first
that $\left( \alpha _{6},\alpha _{9}\right) =\left( 0,0\right) $. Then, we
have $\alpha _{3}\left( \alpha _{4}+\alpha _{8}\right) =0$. Moreover, we have%
\begin{eqnarray*}
\beta _{4} &=&\frac{a_{11}}{a_{33}a_{44}-a_{34}a_{43}}\left( \alpha
_{4}a_{33}a_{44}+\alpha _{5}a_{43}a_{44}-\alpha _{7}a_{33}a_{34}-\alpha
_{8}a_{34}a_{43}\right) , \\
\beta _{5} &=&\frac{a_{11}}{a_{33}a_{44}-a_{34}a_{43}}\left( \alpha
_{5}a_{44}^{2}-\alpha _{7}a_{34}^{2}+\alpha _{4}a_{34}a_{44}-\alpha
_{8}a_{34}a_{44}\right) , \\
\beta _{7} &=&-\frac{a_{11}}{a_{33}a_{44}-a_{34}a_{43}}\left( \alpha
_{5}a_{43}^{2}-\alpha _{7}a_{33}^{2}+\alpha _{4}a_{33}a_{43}-\alpha
_{8}a_{33}a_{43}\right) , \\
\beta _{8} &=&-\frac{a_{11}}{a_{33}a_{44}-a_{34}a_{43}}\left( \alpha
_{4}a_{34}a_{43}+\alpha _{5}a_{43}a_{44}-\alpha _{7}a_{33}a_{34}-\alpha
_{8}a_{33}a_{44}\right) .
\end{eqnarray*}%
Whence $%
\begin{pmatrix}
\beta _{4} & \beta _{5} \\ 
\beta _{7} & \beta _{8}%
\end{pmatrix}%
=a_{11}%
\begin{pmatrix}
a_{33} & a_{34} \\ 
a_{43} & a_{44}%
\end{pmatrix}%
^{-1}%
\begin{pmatrix}
\alpha _{4} & \alpha _{5} \\ 
\alpha _{7} & \alpha _{8}%
\end{pmatrix}%
\begin{pmatrix}
a_{33} & a_{34} \\ 
a_{43} & a_{44}%
\end{pmatrix}%
$. Thus we may assume:%
\begin{equation*}
\begin{pmatrix}
\alpha _{4} & \alpha _{5} \\ 
\alpha _{7} & \alpha _{8}%
\end{pmatrix}%
\in \left\{ 
\begin{pmatrix}
0 & 0 \\ 
0 & 0%
\end{pmatrix}%
,%
\begin{pmatrix}
1 & 0 \\ 
0 & \alpha%
\end{pmatrix}%
,%
\begin{pmatrix}
1 & 1 \\ 
0 & 1%
\end{pmatrix}%
,%
\begin{pmatrix}
0 & 1 \\ 
0 & 0%
\end{pmatrix}%
\right\} .
\end{equation*}%
So, the following cases arises:

\begin{itemize}
\item $%
\begin{pmatrix}
\alpha _{4} & \alpha _{5} \\ 
\alpha _{7} & \alpha _{8}%
\end{pmatrix}%
=%
\begin{pmatrix}
0 & 0 \\ 
0 & 0%
\end{pmatrix}%
$.

\begin{itemize}
\item[\ding{118}] $\alpha _{3}=0$. If $\left( \alpha _{1},\alpha _{2}\right) =\left(
0,0\right) $, then $\theta =0$ and we get the algebra $\mathcal{P}_{4,1}$.
Otherwise, let $\phi $ be the first of the following matrices if $\alpha
_{2}=0$ or the second if $\alpha _{2}\neq 0$:%
\begin{equation*}
\begin{pmatrix}
\alpha _{1} & 0 & 0 & 0 \\ 
0 & \alpha _{1}^{2} & 0 & 0 \\ 
0 & 0 & 1 & 0 \\ 
0 & 0 & 0 & 1%
\end{pmatrix}%
,%
\begin{pmatrix}
1 & 0 & 0 & 0 \\ 
0 & 1 & 0 & 0 \\ 
0 & 0 & 0 & 1 \\ 
0 & 0 & \frac{1}{\alpha _{2}} & -\frac{\alpha _{1}}{\alpha _{2}}%
\end{pmatrix}%
.
\end{equation*}%
Then $\theta \ast \phi =\left( 0,\Delta _{1,3},0,0\right) $. So we obtain
the representative $\left( 0,\Delta _{1,3},0,0\right) $. Hence we get the
Poisson algebra $\mathcal{P}_{4,2}$.

\item[\ding{118}] $\alpha _{3}\neq 0$. We define $\phi $ to be the following
automorphism:%
\begin{equation*}
\phi =%
\begin{pmatrix}
1 & 0 & 0 & 0 \\ 
0 & 1 & 0 & 0 \\ 
-\frac{\alpha _{2}}{\alpha _{3}} & 0 & 1 & 0 \\ 
\frac{\alpha _{1}}{\alpha _{3}} & 0 & 0 & \frac{1}{\alpha _{3}}%
\end{pmatrix}%
.
\end{equation*}%
Then $\theta \ast \phi =\left( 0,\Delta _{3,4},0,0\right) $. So we get the
Poisson algebra $\mathcal{P}_{4,3}$.
\end{itemize}

\item $%
\begin{pmatrix}
\alpha _{4} & \alpha _{5} \\ 
\alpha _{7} & \alpha _{8}%
\end{pmatrix}%
=%
\begin{pmatrix}
1 & 0 \\ 
0 & \alpha%
\end{pmatrix}%
$.

\begin{itemize}
\item[\ding{118}] $\alpha _{3}=0$. If $\alpha \neq 0$, we choose $\phi $ as follows:%
\begin{equation*}
\phi =%
\begin{pmatrix}
1 & 0 & 0 & 0 \\ 
0 & 1 & \alpha _{1} & \frac{\alpha _{2}}{\alpha } \\ 
0 & 0 & 1 & 0 \\ 
0 & 0 & 0 & 1%
\end{pmatrix}%
.
\end{equation*}%
Then $\theta \ast \phi =\left( 0,0,\Delta _{1,3},\alpha \Delta _{1,4}\right) 
$. Hence we get the Poisson algebras $\mathcal{P}_{4,4}^{\alpha \neq 0}$.
Furthermore, the Poisson algebras $\mathcal{P}_{4,4}^{\alpha \neq 0}$ and $%
\mathcal{P}_{4,4}^{\beta \neq 0}$ are isomorphic if and only if $\left(
\alpha -\beta \right) \left( \alpha \beta -1\right) =0$. If $\alpha =0$, we
define $\phi $ to be the first of the following matrices if $\alpha _{2}=0$
or the second if $\alpha _{2}\neq 0$: 
\begin{equation*}
\begin{pmatrix}
1 & 0 & 0 & 0 \\ 
0 & 1 & \alpha _{1} & 0 \\ 
0 & 0 & 1 & 0 \\ 
0 & 0 & 0 & 1%
\end{pmatrix}%
,%
\begin{pmatrix}
1 & 0 & 0 & 0 \\ 
0 & 1 & \alpha _{1} & 0 \\ 
0 & 0 & 1 & 0 \\ 
0 & 0 & 0 & \frac{1}{\alpha _{2}}%
\end{pmatrix}%
.
\end{equation*}%
Then $\theta \ast \phi =\left( 0,0,\Delta _{1,3},0\right) $ if $\alpha
_{2}=0 $ or $\theta \ast \phi =\left( 0,\Delta _{1,4},\Delta _{1,3},0\right) 
$ if $\alpha _{2}\neq 0$. So we get the Poisson algebra $\mathcal{P}%
_{4,4}^{\alpha =0}$ if $\alpha _{2}=0$ or the Poisson algebra $\mathcal{P}%
_{4,5}$ if $\alpha _{2}\neq 0$.

\item[\ding{118}] $\alpha _{3}\neq 0$. Since $\alpha _{3}\left( \alpha _{4}+\alpha
_{8}\right) =0$, we have $\alpha =-1$. Choose $\phi $ as follows:%
\begin{equation*}
\phi =\frac{1}{\alpha _{3}}\allowbreak 
\begin{pmatrix}
\alpha _{3} & 0 & 0 & 0 \\ 
0 & \alpha _{3} & 0 & 0 \\ 
-\alpha _{2} & 0 & \alpha _{3} & 0 \\ 
\alpha _{1} & 0 & 0 & 1%
\end{pmatrix}%
.
\end{equation*}%
Then $\theta \ast \phi =\left( 0,\Delta _{3,4},\Delta _{1,3},-\Delta
_{1,4}\right) $. So we get the Poisson algebra $\mathcal{P}_{4,6}$.
\end{itemize}

\item $%
\begin{pmatrix}
\alpha _{4} & \alpha _{5} \\ 
\alpha _{7} & \alpha _{8}%
\end{pmatrix}%
=%
\begin{pmatrix}
1 & 1 \\ 
0 & 1%
\end{pmatrix}%
$.

Since $\alpha _{3}\left( \alpha _{4}+\alpha _{8}\right) =0$, we have $\alpha
_{3}=0$. Let $\phi $ be the following matrix:%
\begin{equation*}
\phi =%
\begin{pmatrix}
1 & 0 & 0 & 0 \\ 
0 & 1 & \alpha _{1} & \alpha _{2}-\alpha _{1} \\ 
0 & 0 & 1 & 0 \\ 
0 & 0 & 0 & 1%
\end{pmatrix}%
.
\end{equation*}%
Then $\theta \ast \phi =\left( 0,0,\Delta _{1,3}+\Delta _{1,4},\Delta
_{1,4}\right) $. Therefore, we get the Poisson algebra $\mathcal{P}_{4,7}$.

\item $%
\begin{pmatrix}
\alpha _{4} & \alpha _{5} \\ 
\alpha _{7} & \alpha _{8}%
\end{pmatrix}%
=%
\begin{pmatrix}
0 & 1 \\ 
0 & 0%
\end{pmatrix}%
$.

\begin{itemize}
\item[\ding{118}] $\alpha _{3}=0$. Let $\phi $ be the first of the following matrices if 
$\alpha _{1}=0$ or the second if $\alpha _{1}\neq 0$:%
\begin{equation*}
\begin{pmatrix}
1 & 0 & 0 & 0 \\ 
0 & 1 & \alpha _{2} & 0 \\ 
0 & 0 & 1 & 0 \\ 
0 & 0 & 0 & 1%
\end{pmatrix}%
,%
\begin{pmatrix}
\alpha _{1} & 0 & 0 & 0 \\ 
0 & \alpha _{1}^{2} & \alpha _{2} & 0 \\ 
0 & 0 & 1 & 0 \\ 
0 & 0 & 0 & \frac{1}{\alpha _{1}}%
\end{pmatrix}%
.
\end{equation*}%
Then $\theta \ast \phi =\left( 0,0,\Delta _{1,4},0\right) $ if $\alpha
_{1}=0 $ or $\theta \ast \phi =\left( 0,\Delta _{1,3},\Delta _{1,4},0\right) 
$ if $\alpha _{1}\neq 0$. Therefore we obtain the Poisson algebras $\mathcal{%
P}_{4,8}$ and $\mathcal{P}_{4,9}$.

\item[\ding{118}] $\alpha _{3}\neq 0$. Choose $\phi $ as follows:%
\begin{equation*}
\phi =%
\begin{pmatrix}
\alpha _{3} & 0 & 0 & 0 \\ 
0 & \alpha _{3}^{2} & \alpha _{2}\alpha _{3} & 0 \\ 
0 & 0 & \alpha _{3} & 0 \\ 
\alpha _{1} & 0 & 0 & 1%
\end{pmatrix}%
.
\end{equation*}%
Then $\theta \ast \phi =\left( 0,\Delta _{3,4},\Delta _{1,4},0\right) $ and
we get the algebra $\mathcal{P}_{4,10}$.
\end{itemize}
\end{itemize}

Let us assume now that $\left( \alpha _{6},\alpha _{9}\right) =\left(
1,0\right) $. Then $\alpha _{7}=\alpha _{8}=0$ and $\alpha _{1}=\alpha
_{3}\alpha _{4}$. Set $\lambda =\alpha _{2}-\alpha _{3}\alpha _{5}$. Let $%
\phi $ be the first of the following matrices if $\lambda =0$ or the second
if $\lambda \neq 0$:%
\begin{equation*}
\begin{pmatrix}
1 & 0 & 0 & 0 \\ 
0 & 1 & \alpha _{3} & 0 \\ 
-\alpha _{5} & 0 & 1 & 0 \\ 
\alpha _{4} & 0 & 0 & 1%
\end{pmatrix}%
,%
\begin{pmatrix}
\lambda & 0 & 0 & 0 \\ 
0 & \lambda ^{2} & \alpha _{3} & 0 \\ 
-\lambda \alpha _{5} & 0 & 1 & 0 \\ 
\lambda \alpha _{4} & 0 & 0 & 1%
\end{pmatrix}%
.
\end{equation*}%
Then $\theta \ast \phi =\left( 0,0,\Delta _{3,4},0\right) $ if $\lambda =0$
or $\theta \ast \phi =\left( 0,\Delta _{1,4},\Delta _{3,4},0\right) $ if $%
\lambda \neq 0$. Thus, we obtain the Poisson algebras $\mathcal{P}_{4,11}$
and $\mathcal{P}_{4,12}$.

\underline{$\left( \mathcal{P},\cdot \right) =\mathcal{A}_{03}$.} Let $%
\theta =\left( B_{1},B_{2},B_{3},B_{4}\right) $ be an arbitrary element of $%
Z^{2}\left( \mathcal{P},\mathcal{P}\right) $. Then%
\begin{eqnarray*}
B_{1} &=&B_{2}=0, \\
B_{3} &=&\alpha _{1}\Delta _{1,2}+\alpha _{2}\Delta _{1,4}+\alpha _{3}\Delta
_{2,4}, \\
B_{4} &=&\alpha _{4}\Delta _{1,2}+\alpha _{5}\Delta _{1,4}+\alpha _{6}\Delta
_{2,4},
\end{eqnarray*}%
such that%
\begin{equation*}
\alpha _{2}\alpha _{6}-\alpha _{3}\alpha _{5}=0,
\end{equation*}%
for some $\alpha _{1},\ldots ,\alpha _{6}\in \mathbb{C}$. Moreover, the
automorphism group of $\mathcal{A}_{03}$, $\text{Aut}\left( \mathcal{A}%
_{03}\right) $, consists of the automorphisms $\phi $ given by a matrix of
the following form:%
\begin{equation*}
\begin{pmatrix}
a_{11} & \varepsilon a_{21} & 0 & 0 \\ 
a_{21} & \epsilon a_{11} & 0 & 0 \\ 
a_{31} & a_{32} & a_{11}^{2}+a_{21}^{2} & a_{34} \\ 
a_{41} & a_{42} & 0 & a_{44}%
\end{pmatrix}%
:\left( \varepsilon ,\epsilon \right) \in \left\{ \left( 1,-1\right) ,\left(
-1,1\right) \right\} .
\end{equation*}%
Let $\phi =\bigl(a_{ij}\bigr)\in $ $\text{Aut}\left( \mathcal{A}_{03}\right) 
$. Then $\theta \ast \phi =\left( 0,0,\beta _{1}\Delta _{1,2}+\beta
_{2}\Delta _{1,4}+\beta _{3}\Delta _{2,4},\beta _{4}\Delta _{1,2}+\beta
_{5}\Delta _{1,4}+\beta _{6}\Delta _{2,4}\right) $ where%
\begin{eqnarray*}
\beta _{1} &=&\frac{1}{a_{44}\left( a_{11}^{2}+a_{21}^{2}\right) }\left( 
\begin{array}{c}
\alpha _{2}a_{11}a_{42}a_{44}+\alpha _{3}a_{21}a_{42}a_{44}-\alpha
_{5}a_{11}a_{42}a_{34}-\alpha _{6}a_{21}a_{42}a_{34}+\epsilon \alpha
_{1}a_{11}^{2}a_{44}-\varepsilon \alpha _{1}a_{21}^{2}a_{44} \\ 
-\epsilon \alpha _{4}a_{11}^{2}a_{34}+\varepsilon \alpha
_{4}a_{21}^{2}a_{34}-\epsilon \alpha _{3}a_{11}a_{41}a_{44}-\varepsilon
\alpha _{2}a_{21}a_{41}a_{44}+\epsilon \alpha
_{6}a_{11}a_{41}a_{34}+\varepsilon \alpha _{5}a_{21}a_{41}a_{34}%
\end{array}%
\right) , \\
\beta _{2} &=&\frac{1}{a_{11}^{2}+a_{21}^{2}}\left( \alpha
_{2}a_{11}a_{44}+\alpha _{3}a_{21}a_{44}-\alpha _{5}a_{11}a_{34}-\alpha
_{6}a_{21}a_{34}\right) , \\
\beta _{3} &=&\frac{1}{a_{11}^{2}+a_{21}^{2}}\left( \epsilon \alpha
_{3}a_{11}a_{44}+\varepsilon \alpha _{2}a_{21}a_{44}-\epsilon \alpha
_{6}a_{11}a_{34}-\varepsilon \alpha _{5}a_{21}a_{34}\right) , \\
\beta _{4} &=&\frac{1}{a_{44}}\left( \epsilon \alpha
_{4}a_{11}^{2}-\varepsilon \alpha _{4}a_{21}^{2}+\alpha
_{5}a_{11}a_{42}+\alpha _{6}a_{21}a_{42}-\epsilon \alpha
_{6}a_{11}a_{41}-\varepsilon \alpha _{5}a_{21}a_{41}\right) , \\
\beta _{5} &=&\alpha _{5}a_{11}+\alpha _{6}a_{21}, \\
\beta _{6} &=&\epsilon \alpha _{6}a_{11}+\varepsilon \alpha _{5}a_{21}.
\end{eqnarray*}%
Assume first that $\alpha _{5}^{2}+\alpha _{6}^{2}\neq 0$. Let $\phi $ be
the following automorphism:%
\begin{equation*}
\phi =\frac{1}{\alpha _{5}^{2}+\alpha _{6}^{2}}\allowbreak 
\begin{pmatrix}
\alpha _{5} & \alpha _{6} & 0 & 0 \\ 
\alpha _{6} & -\alpha _{5} & 0 & 0 \\ 
0 & 0 & 1 & \alpha _{2}\alpha _{5}+\alpha _{3}\alpha _{6} \\ 
0 & \alpha _{4} & 0 & \alpha _{5}^{2}+\alpha _{6}^{2}%
\end{pmatrix}%
.
\end{equation*}%
Then $\theta \ast \phi =\left( 0,0,\alpha \Delta _{1,2},\Delta _{1,4}\right) 
$ for some $\alpha \in \mathbb{C}$. So we get the representatives $\theta
^{\alpha }=\left( 0,0,\alpha \Delta _{1,2},\Delta _{1,4}\right) $. Further,
the representatives $\theta ^{\alpha }$ and $\theta ^{\beta }$ are in the
same orbit if and only if $\alpha ^{2}=\beta ^{2}$. So we get the Poisson
algebras $\mathcal{P}_{4,13}^{\alpha }$. Assume now that $\alpha
_{5}^{2}+\alpha _{6}^{2}=0$ (i.e. $\alpha _{6}=\pm i\alpha _{5}$ where $i=%
\sqrt{-1}$).

\begin{itemize}
\item $\alpha _{5}=0$. Then so is $\alpha _{6}=0$. Assume first that $\alpha
_{2}^{2}+\alpha _{3}^{2}\neq 0$. If $\alpha _{4}=0$, we define $\phi $ to be
the following automorphism: 
\begin{equation*}
\phi =\frac{1}{\alpha _{2}^{2}+\alpha _{3}^{2}}\allowbreak 
\begin{pmatrix}
\alpha _{2} & \alpha _{3} & 0 & 0 \\ 
\alpha _{3} & -\alpha _{2} & 0 & 0 \\ 
0 & 0 & 1 & 0 \\ 
0 & \alpha _{1} & 0 & 1%
\end{pmatrix}%
.
\end{equation*}%
Then $\theta \ast \phi =\left( 0,0,\Delta _{1,4},0\right) $. So we get the
Poisson algebra $\mathcal{P}_{4,14}$. If $\alpha _{4}\neq 0$, we define $%
\phi $ to be the following automorphism: 
\begin{equation*}
\phi =\frac{1}{\alpha _{4}^{2}\left( \alpha _{2}^{2}+\alpha _{3}^{2}\right) }%
\allowbreak 
\begin{pmatrix}
\alpha _{2}\alpha _{4} & -\alpha _{3}\alpha _{4} & 0 & 0 \\ 
\alpha _{3}\alpha _{4} & \alpha _{2}\alpha _{4} & 0 & 0 \\ 
0 & 0 & 1 & 0 \\ 
0 & -\alpha _{1}\alpha _{4} & 0 & \alpha _{4}%
\end{pmatrix}%
.
\end{equation*}%
Then $\theta \ast \phi =\left( 0,0,\Delta _{1,4},\Delta _{1,2}\right) $. So
we get the Poisson algebra $\mathcal{P}_{4,15}$. Assume now that $\alpha
_{3}=\pm i\alpha _{2}$ where $i=\sqrt{-1}$.

\begin{itemize}
\item[\ding{118}] $\alpha _{2}=0$. Then so is $\alpha _{3}=0$. If $\alpha _{4}=0$, then $%
\theta =\left( 0,0,\alpha \Delta _{1,2},0\right) $ for some $\alpha \in 
\mathbb{C}$. So we have the representatives $\vartheta ^{\alpha }=\left(
0,0,\alpha \Delta _{1,2},0\right) $. Moreover, for any $\phi =\bigl(a_{ij}\bigr)%
\in $ $\text{Aut}\left( \mathcal{A}_{03}\right) $, we have $\vartheta
^{\alpha }\ast \phi =\left( 0,0,\beta \Delta _{1,2},0\right) $ with $%
\alpha ^{2}=\beta ^{2}$. Thus the representatives $\vartheta ^{\alpha
},\vartheta ^{\beta }$ are in the same orbit if and only if $\alpha
^{2}=\beta ^{2}$. Hence we get the algebras $\mathcal{P}_{4,16}^{\alpha }$.
If $\alpha _{4}\neq 0$, we choose $\phi $ as follows:%
\begin{equation*}
\phi =%
\begin{pmatrix}
1 & 0 & 0 & 0 \\ 
0 & 1 & 0 & 0 \\ 
0 & 0 & 1 & \alpha _{1} \\ 
0 & 0 & 0 & \alpha _{4}%
\end{pmatrix}%
.
\end{equation*}%
Then $\theta \ast \phi =\left( 0,0,0,\Delta _{1,2}\right) $. So we get the
Poisson algebra $\mathcal{P}_{4,17}$.

\item[\ding{118}] $\alpha _{2}\neq 0$. Then so is $\alpha _{3}\neq 0$. If $\alpha _{4}=0$%
, we choose $\phi $ to be the first of the following matrices when $\alpha
_{3}=i\alpha _{2}$ or the second when $\alpha _{3}=-i\alpha _{2}$:%
\begin{equation*}
\begin{pmatrix}
\alpha _{2} & 0 & 0 & 0 \\ 
0 & \alpha _{2} & 0 & 0 \\ 
0 & 0 & \alpha _{2}^{2} & 0 \\ 
-i\alpha_{1} & 0 & 0 & 1%
\end{pmatrix}%
,%
\begin{pmatrix}
\alpha _{2} & 0 & 0 & 0 \\ 
0 & -\alpha _{2} & 0 & 0 \\ 
0 & 0 & \alpha _{2}^{2} & 0 \\ 
i\alpha_{1} & 0 & 0 & 1%
\end{pmatrix}%
\allowbreak .
\end{equation*}%
Then $\theta \ast \phi =\left( 0,0,\Delta _{1,4}+i\Delta _{2,4},0\right) $.
So we get the Poisson algebra $\mathcal{P}_{4,18}$.

If $\alpha _{4}\neq 0$, we choose $\phi $ to be the first of the following
matrices if $\alpha _{3}=i\alpha _{2}$ or the second if $\alpha
_{3}=-i\alpha _{2}$:%
\begin{equation*}
\frac{1}{\alpha _{2}^{2}\alpha _{4}^{2}}\allowbreak 
\begin{pmatrix}
\alpha _{2}\alpha _{4} & 0 & 0 & 0 \\ 
0 & \alpha _{2}\alpha _{4} & 0 & 0 \\ 
0 & 0 & 1 & \alpha _{1} \\ 
0 & 0 & 0 & \alpha _{4}%
\end{pmatrix}%
,\frac{1}{\alpha _{2}^{2}\alpha _{4}^{2}}\allowbreak 
\begin{pmatrix}
-\alpha _{2}\alpha _{4} & 0 & 0 & 0 \\ 
0 & \alpha _{2}\alpha _{4} & 0 & 0 \\ 
0 & 0 & 1 & -\alpha _{1} \\ 
0 & 0 & 0 & -\alpha _{4}%
\end{pmatrix}%
.\allowbreak \allowbreak
\end{equation*}%
Then $\theta \ast \phi =\left( 0,0,\Delta _{1,4}+i\Delta _{2,4},\Delta
_{1,2}\right) $. So we get the Poisson algebra $\mathcal{P}_{4,19}$.
\end{itemize}

\item $\alpha _{5}\neq 0$. Then so is $\alpha _{6}\neq 0$. Let $\phi $ be
the first of the following matrices when $\alpha _{6}=i\alpha _{5}$ or the
second when $\alpha _{6}=-i\alpha _{5}$:%
\begin{equation*}
\frac{1}{\alpha _{5}^{2}}\allowbreak 
\begin{pmatrix}
\alpha _{5} & 0 & 0 & 0 \\ 
0 & \alpha _{5} & 0 & 0 \\ 
0 & 0 & 1 & \alpha _{2}\alpha _{5} \\ 
0 & -\alpha _{4} & 0 & \alpha _{5}^{2}%
\end{pmatrix}%
,\frac{1}{\alpha _{5}^{2}}\allowbreak 
\begin{pmatrix}
\alpha _{5} & 0 & 0 & 0 \\ 
0 & -\alpha _{5} & 0 & 0 \\ 
0 & 0 & 1 & \alpha _{2}\alpha _{5} \\ 
0 & \alpha _{4} & 0 & \alpha _{5}^{2}%
\end{pmatrix}%
.
\end{equation*}%
Then $\theta \ast \phi =\left( 0,0,\alpha \Delta _{1,2},\Delta
_{1,4}+i\Delta _{2,4}\right) $ for some $\alpha \in \mathbb{C}$. So we have
the representatives $\eta ^{\alpha }=\left( 0,0,\alpha \Delta _{1,2},\Delta
_{1,4}+i\Delta _{2,4}\right) $. Moreover, the representatives $\eta ^{\alpha
},\eta ^{\beta }$ are in the same orbit if and only if $\alpha ^{2}=\beta
^{2}$. So we get the Poisson algebras $\mathcal{P}_{4,20}^{\alpha }$.
\end{itemize}

\underline{$\left( \mathcal{P},\cdot \right) =\mathcal{A}_{04}$.} Let $%
\theta =\left( B_{1},B_{2},B_{3},B_{4}\right) $ be an arbitrary element of $%
Z^{2}\left( \mathcal{P},\mathcal{P}\right) $. Then $\theta =\left(
0,0,\alpha _{1}\Delta _{1,4},\alpha _{2}\Delta _{1,4}\right) $ for some $%
\alpha _{1},\alpha _{2}\in \mathbb{C}$. The automorphism group of $\mathcal{A%
}_{04}$, $\text{Aut}\left( \mathcal{A}_{04}\right) $, consists of the
automorphisms $\phi $ given by a matrix of the following form:%
\begin{equation*}
\begin{pmatrix}
a_{11} & 0 & 0 & 0 \\ 
a_{21} & a_{11}^{2} & 0 & 0 \\ 
a_{31} & 2a_{11}a_{21} & a_{11}^{3} & a_{34} \\ 
a_{41} & 0 & 0 & a_{44}%
\end{pmatrix}%
.
\end{equation*}%
Let $\phi =\bigl(a_{ij}\bigr)\in $ $\text{Aut}\left( \mathcal{A}_{04}\right) 
$. Then $\theta \ast \phi =\left( 0,0,\beta _{1}\Delta _{1,4},\beta
_{2}\Delta _{1,4}\right) $ where%
\begin{eqnarray*}
\beta _{1} &=&\frac{1}{a_{11}^{2}}\left( \alpha _{1}a_{44}-\alpha
_{2}a_{34}\right) , \\
\beta _{2} &=&\alpha _{2}a_{11}.
\end{eqnarray*}%
If $\theta =0$, we get the Poisson
algebra $\mathcal{P}_{4,21}$. Assume now that $\theta \neq0$. If $\alpha
_{2}\neq 0$, we choose $\phi $ as follows:%
\begin{equation*}
\phi =\frac{1}{\alpha _{2}^{3}}%
\begin{pmatrix}
\alpha _{2}^{2} & 0 & 0 & 0 \\ 
0 & \alpha _{2} & 0 & 0 \\ 
0 & 0 & 1 & \alpha _{1}\alpha _{2}^{2} \\ 
0 & 0 & 0 & \alpha _{2}^{3}%
\end{pmatrix}%
.
\end{equation*}%
Then $\theta \ast \phi =\left( 0,0,0,\Delta _{1,4}\right) $. Hence we get
the Poisson algebra $\mathcal{P}_{4,22}$. If $\alpha _{2}=0$, we choose $%
\phi $ as follows:%
\begin{equation*}
\phi =%
\begin{pmatrix}
1 & 0 & 0 & 0 \\ 
0 & 1 & 0 & 0 \\ 
0 & 0 & 1 & 0 \\ 
0 & 0 & 0 & \frac{1}{\alpha _{1}}%
\end{pmatrix}%
.
\end{equation*}%
Then $\theta \ast \phi =\left( 0,0,\Delta _{1,4},0\right) $. So we get the
Poisson algebra $\mathcal{P}_{4,23}$.

\underline{$\left( \mathcal{P},\cdot \right) =\mathcal{A}_{05}$.} It is easy
to see that the algebra $\mathcal{A}_{05}$ is isomorphic to the following
algebra%
\begin{equation*}
\mathcal{A}_{05}^{\prime }:e_{1}^{2}=e_{3},e_{2}^{2}=e_{4}.
\end{equation*}%
So we may assume $\left( \mathcal{P},\cdot \right) =\mathcal{A}_{05}^{\prime
}$. Let $\theta =\left( B_{1},B_{2},B_{3},B_{4}\right) $ be an arbitrary
element of $Z^{2}\left( \mathcal{P},\mathcal{P}\right) $. Then $\theta
=\left( 0,0,\alpha _{1}\Delta _{1,2},\alpha _{2}\Delta _{1,2}\right) $ for
some $\alpha _{1},\alpha _{2}\in \mathbb{C}$. The automorphism group of $%
\mathcal{A}_{05}^{\prime }$, $\text{Aut}\left( \mathcal{A}_{05}^{\prime
}\right) $, consists of the automorphisms $\phi $ given by a matrix of the
following form:%
\begin{equation*}
\begin{pmatrix}
\epsilon a_{11} & \left( 1-\epsilon \right) a_{12} & 0 & 0 \\ 
\left( 1-\epsilon \right) a_{21} & \epsilon a_{22} & 0 & 0 \\ 
a_{31} & a_{32} & \epsilon a_{11}^{2} & \left( 1-\epsilon \right) a_{12}^{2}
\\ 
a_{41} & a_{42} & \left( 1-\epsilon \right) a_{21}^{2} & \epsilon a_{22}^{2}%
\end{pmatrix}%
:\epsilon \in \left\{ 0,1\right\} .
\end{equation*}%
Let $\phi =\bigl(a_{ij}\bigr)\in $ $\text{Aut}\left( \mathcal{A}%
_{05}^{\prime }\right) $. Then $\theta \ast \phi =\left( 0,0,\beta
_{1}\Delta _{1,2},\beta _{2}\Delta _{1,2}\right) $ where $\beta _{1}=\frac{%
a_{22}}{a_{11}}\alpha _{1},\beta _{2}=\alpha _{2}\frac{a_{11}}{a_{22}}$ if $%
\epsilon =1$ or $\beta _{1}=-\alpha _{2}\frac{a_{12}}{a_{21}},\beta _{2}=-%
\frac{\alpha _{1}}{a_{12}}a_{21}$ if $\epsilon =0$. If $\theta =0$, we get
the Poisson algebra $\mathcal{P}_{4,24}$. If $\theta \neq 0$, then we may
assume without any loss of generality that $\alpha _{1}\neq 0$. Let $\phi $
be the following automorphism:%
\begin{equation*}
\phi =%
\begin{pmatrix}
\alpha _{1} & 0 & 0 & 0 \\ 
0 & 1 & 0 & 0 \\ 
0 & 0 & \alpha _{1}^{2} & 0 \\ 
0 & 0 & 0 & 1%
\end{pmatrix}%
.
\end{equation*}%
Then $\theta \ast \phi= \left( 0,0,\Delta
_{1,2},\alpha \Delta _{1,2}\right) $ for some $\alpha \in \mathbb{C}$. Hence we get the representatives $\theta ^{\alpha }=\left( 0,0,\Delta
_{1,2},\alpha \Delta _{1,2}\right) $. Moreover, $\theta ^{\alpha }$ and $%
\theta ^{\beta }$ are in the same orbit if and only if $\alpha =\beta $.
Hence we get the Poisson algebras $\mathcal{P}_{4,25}^{\alpha }$.

\underline{$\left( \mathcal{P},\cdot \right) =\mathcal{A}_{06}$.} Let $%
\theta =\left( B_{1},B_{2},B_{3},B_{4}\right) $ be an arbitrary element of $%
Z^{2}\left( \mathcal{P},\mathcal{P}\right) $. Then 
\begin{equation*}
\theta =(\alpha _{1}\Delta _{1,2},0,\alpha _{2}\Delta _{1,2},\alpha
_{3}\Delta _{1,2}+2\alpha _{1}\Delta _{1,3}-\alpha _{1}\Delta _{2,4})
\end{equation*}%
for some $\alpha _{1},\alpha _{2},\alpha _{3}\in \mathbb{C}$. The
automorphism group of $\mathcal{A}_{06}$, $\text{Aut}\left( \mathcal{A}%
_{06}\right) $, consists of the automorphisms $\phi $ given by a matrix of
the following form:%
\begin{equation*}
\begin{pmatrix}
a_{11} & a_{12} & 0 & 0 \\ 
0 & a_{22} & 0 & 0 \\ 
a_{31} & a_{32} & a_{22}^{2} & 0 \\ 
a_{41} & a_{42} & 2a_{12}a_{22} & a_{11}a_{22}%
\end{pmatrix}%
.
\end{equation*}%
Let $\phi =\bigl(a_{ij}\bigr)\in $Aut$\left( \mathcal{A}_{06}\right) $. Then 
$\theta \ast \phi =(\beta _{1}\Delta _{1,2},0,\beta _{2}\Delta _{1,2},\beta
_{3}\Delta _{1,2}+2\beta _{1}\Delta _{1,3}-\beta _{1}\Delta _{2,4})$ where%
\begin{eqnarray*}
\beta _{1} &=&\alpha _{1}a_{22}, \\
\beta _{2} &=&\frac{1}{a_{22}}\left( \alpha _{2}a_{11}-\alpha
_{1}a_{31}\right) , \\
\beta _{3} &=&\frac{1}{a_{22}}\left( 2\alpha _{1}a_{32}-2\alpha
_{2}a_{12}+\alpha _{3}a_{22}\right) .
\end{eqnarray*}%
Let us consider the following cases:$\allowbreak $

\begin{itemize}
\item $\alpha _{1}\neq 0$. Let $\phi $ be the following automorphism:%
\begin{equation*}
\phi =%
\begin{pmatrix}
1 & 0 & 0 & 0 \\ 
0 & \frac{1}{\alpha _{1}} & 0 & 0 \\ 
\frac{\alpha _{2}}{\alpha _{1}} & -\frac{1}{2\alpha _{1}^{2}}\alpha _{3} & 
\frac{1}{\alpha _{1}^{2}} & 0 \\ 
0 & 0 & 0 & \frac{1}{\alpha _{1}}%
\end{pmatrix}%
.
\end{equation*}%
Then $\theta \ast \phi =(\Delta _{1,2},0,0,2\Delta _{1,3}-\Delta _{2,4})$.
So we obtain the Poisson algebra $\mathcal{P}_{4,26}$.$\allowbreak $

\item $\alpha _{1}=0,\alpha _{2}\neq 0$. Let $\phi $ be the following
automorphism:%
\begin{equation*}
\phi =%
\begin{pmatrix}
1 & \frac{1}{2}\alpha _{3} & 0 & 0 \\ 
0 & \alpha _{2} & 0 & 0 \\ 
0 & 0 & \alpha _{2}^{2} & 0 \\ 
0 & 0 & \alpha _{2}\alpha _{3} & \alpha _{2}%
\end{pmatrix}%
.
\end{equation*}%
Then $\theta \ast \phi =(0,0,\Delta _{1,2},0)$. So we obtain the Poisson
algebra $\mathcal{P}_{4,27}$.$\allowbreak $

\item $\alpha _{1}=\alpha _{2}=0$. Then $\theta \ast \phi =\theta $ for any $%
\phi \in $Aut$\left( \mathcal{A}_{06}\right) $. Thus we have the
representatives $\theta ^{\alpha }=\left( 0,0,0,\alpha \Delta _{1,2}\right) $
and $\theta ^{\alpha },$ $\theta ^{\beta }$ are in the same orbit if and
only if $\alpha =\beta $. Therefore, we obtain the Poisson algebras $%
\mathcal{P}_{4,28}^{\alpha }$.
\end{itemize}

\underline{$\left( \mathcal{P},\cdot \right) =\mathcal{A}_{07}$.} Since $%
\mathcal{A}_{07}$ is isomorphic to $\mathcal{A}_{07}^{\prime }:e_{1}\cdot
e_{2}=e_{4},e_{3}^{2}=e_{4}$, we may assume $\left( \mathcal{P},\cdot
\right) =\mathcal{A}_{07}^{\prime }$. Let $\theta =\left(
B_{1},B_{2},B_{3},B_{4}\right) $ be an arbitrary element of $Z^{2}\left( 
\mathcal{P},\mathcal{P}\right) $. Then%
\begin{equation*}
\theta =(-\alpha _{1}\Delta _{1,3},\alpha _{1}\Delta _{2,3},\alpha
_{1}\Delta _{1,2},\alpha _{2}\Delta _{1,2}+\alpha _{3}\Delta _{1,3}+\alpha
_{4}\Delta _{2,3})
\end{equation*}%
for some $\alpha _{1},\ldots ,\alpha _{4}\in \mathbb{C}$. The automorphism
group of $\mathcal{A}_{07}^{\prime }$, $\text{Aut}\left( \mathcal{A}%
_{07}^{\prime }\right) $, consists of the automorphisms $\phi $ given by a
matrix of the following form:%
\begin{equation*}
\phi =%
\begin{pmatrix}
a_{11} & a_{12} & a_{13} & 0 \\ 
a_{21} & a_{22} & a_{23} & 0 \\ 
a_{31} & a_{32} & a_{33} & 0 \\ 
a_{41} & a_{42} & a_{43} & a_{44}%
\end{pmatrix}%
\end{equation*}%
such that%
\begin{equation*}
\begin{pmatrix}
a_{23} & a_{13} & a_{33} \\ 
a_{22} & a_{12} & a_{32} \\ 
a_{21} & a_{11} & a_{31}%
\end{pmatrix}%
\begin{pmatrix}
a_{11} & a_{12} & a_{13} \\ 
a_{21} & a_{22} & a_{23} \\ 
a_{31} & a_{32} & a_{33}%
\end{pmatrix}%
=%
\begin{pmatrix}
0 & 0 & a_{44} \\ 
a_{44} & 0 & 0 \\ 
0 & a_{44} & 0%
\end{pmatrix}%
.
\end{equation*}%
Let us consider the following cases:

\begin{itemize}
\item $\alpha _{1}\neq 0$. Let us define $\phi $ as follows:%
\begin{equation*}
\phi =\frac{1}{\alpha _{1}^{2}}%
\begin{pmatrix}
\alpha _{1} & 0 & 0 & 0 \\ 
0 & \alpha _{1} & 0 & 0 \\ 
0 & 0 & \alpha _{1} & 0 \\ 
-\alpha _{3} & \alpha _{4} & \alpha _{2} & 1%
\end{pmatrix}%
.
\end{equation*}%
Then $\theta \ast \phi =\left( -\Delta _{1,3},\Delta _{2,3},\Delta
_{1,2},0\right) $. So we get the Poisson algebra $\mathcal{P}_{4,29}$.

\item $\alpha _{1}=0,\alpha _{3}\neq 0$. Set $\lambda =\alpha
_{2}^{2}-2\alpha _{3}\alpha _{4}$. Let $\phi $ be the first of the following
matrices if $\lambda =0$ or the second if $\lambda \neq 0$:%
\begin{equation*}
\begin{pmatrix}
\frac{1}{\alpha _{3}} & -\alpha _{4} & \frac{\alpha _{2}}{\alpha _{3}} & 0
\\ 
0 & \alpha _{3} & 0 & 0 \\ 
0 & -\alpha _{2} & 1 & 0 \\ 
0 & 0 & 0 & 1%
\end{pmatrix}%
,%
\begin{pmatrix}
-\frac{1}{\alpha _{3}^{2}}\left( \alpha _{2}^{2}-\alpha _{3}\alpha _{4}+%
\sqrt{\lambda }\alpha _{2}\right) & -\frac{1}{2\lambda }\left( \alpha
_{3}\alpha _{4}-\alpha _{2}^{2}+\sqrt{\lambda }\alpha _{2}\right) & \frac{1}{%
\sqrt{\lambda }}\alpha _{4} & 0 \\ 
1 & -\frac{1}{2\lambda }\alpha _{3}^{2} & -\frac{1}{\sqrt{\lambda }}\alpha
_{3} & 0 \\ 
-\frac{1}{\alpha _{3}}\left( \alpha _{2}+\sqrt{\lambda }\right) & -\frac{1}{%
2\lambda ^{\frac{3}{2}}}\alpha _{3}\left( \lambda -\sqrt{\lambda }\alpha
_{2}\right) & \frac{1}{\sqrt{\lambda }}\alpha _{2} & 0 \\ 
0 & 0 & 0 & 1%
\end{pmatrix}%
.
\end{equation*}%
Then $\theta \ast \phi =\left( 0,0,0,\Delta _{1,3}\right) $ if $\lambda =0$
or $\theta \ast \phi =\left( 0,0,0,\alpha \Delta _{1,2}\right) $ with $%
\alpha \neq 0$ if $\lambda \neq 0$. Thus we obtain the Poisson algebras $%
\mathcal{P}_{4,30}$ and $\mathcal{P}_{4,31}^{\alpha \neq 0}$. Moreover, the
algebras $\mathcal{P}_{4,31}^{\alpha \neq 0}$ and $\mathcal{P}%
_{4,31}^{\alpha \neq 0}$ are isomorphic if and only if $\alpha ^{2}=\beta
^{2}$.

\item $\alpha _{1}=0,\alpha _{3}=0$.

\begin{itemize}
\item[\ding{118}] $\alpha _{2}\neq 0$. If we choose $\phi $ as follows:%
\begin{equation*}
\phi =%
\begin{pmatrix}
1 & -\frac{1}{2\alpha _{2}^{2}}\alpha _{4}^{2} & \frac{1}{\alpha _{2}}\alpha
_{4} & 0 \\ 
0 & 1 & 0 & 0 \\ 
0 & -\frac{1}{\alpha _{2}}\alpha _{4} & 1 & 0 \\ 
0 & 0 & 0 & 1%
\end{pmatrix}%
,
\end{equation*}%
then $\theta \ast \phi =\left( 0,0,0,\alpha _{2}\Delta _{1,2}\right) $ and
so we get again the Poisson algebras $\mathcal{P}_{4,31}^{\alpha \neq 0}$.

\item[\ding{118}] $\alpha _{2}=0$. If $\alpha _{4}=0$, then $\theta =0$ and we get the
algebra $\mathcal{P}_{4,31}^{\alpha =0}$. If $\alpha _{4}\neq 0$, we choose $%
\phi $ to be the following automorphism: 
\begin{equation*}
\phi =%
\begin{pmatrix}
0 & -\alpha _{4} & 0 & 0 \\ 
-\frac{1}{\alpha _{4}} & 0 & 0 & 0 \\ 
0 & 0 & -1 & 0 \\ 
0 & 0 & 0 & 1%
\end{pmatrix}%
.
\end{equation*}%
Then $\theta \ast \phi =\left( 0,0,0,\Delta _{1,3}\right) $ and we have
again the algebra $\mathcal{P}_{4,30}$.
\end{itemize}
\end{itemize}

\underline{$\left( \mathcal{P},\cdot \right) =\mathcal{A}_{08}$.} Let $%
\theta =\left( B_{1},B_{2},B_{3},B_{4}\right) $ be an arbitrary element of $%
Z^{2}\left( \mathcal{P},\mathcal{P}\right) $. Then $\theta =\left(
0,0,0,\alpha \Delta _{1,3}\right) $ for some $\alpha \in \mathbb{C}$. If $%
\alpha =0$, we then get the algebra $\mathcal{P}_{4,32}$. If $\alpha \neq 0$%
, we define $\phi $ to be the following diagonal matrix:%
\begin{equation*}
\phi =%
\begin{pmatrix}
\alpha ^{2} & 0 & 0 & 0 \\ 
0 & \alpha ^{4} & 0 & 0 \\ 
0 & 0 & \alpha ^{3} & 0 \\ 
0 & 0 & 0 & \alpha ^{6}%
\end{pmatrix}%
.
\end{equation*}%
Then $\phi \in \text{Aut}\left( \mathcal{A}_{08}\right) $ and $\theta \ast
\phi =\left( 0,0,0,\Delta _{1,3}\right) $. Hence we get the Poisson algebra $%
\mathcal{P}_{4,33}$.

\underline{$\left( \mathcal{P},\cdot \right) =\mathcal{A}_{09}$.} Then $%
Z^{2}\left( \mathcal{P},\mathcal{P}\right) =\left\{ 0\right\} $. So we get
the algebra $\mathcal{P}_{4,34}$.

\underline{$\left( \mathcal{P},\cdot \right) =\mathcal{A}_{10}$.} Then $%
Z^{2}\left( \mathcal{P},\mathcal{P}\right) =\left\{ 0\right\} $. So we get
the algebra $\mathcal{P}_{4,35}$.

\underline{$\left( \mathcal{P},\cdot \right) =\mathcal{A}_{11}$.} Then $%
Z^{2}\left( \mathcal{P},\mathcal{P}\right) =\left\{ 0\right\} $. So we get
the algebra $\mathcal{P}_{4,36}$.

\underline{$\left( \mathcal{P},\cdot \right) =\mathcal{A}_{12}$.} Then $%
Z^{2}\left( \mathcal{P},\mathcal{P}\right) =\left\{ 0\right\} $. So we get
the algebra $\mathcal{P}_{4,37}$.

\underline{$\left( \mathcal{P},\cdot \right) =\mathcal{A}_{13}$.} Let $%
\theta =\left( B_{1},B_{2},B_{3},B_{4}\right) $ be an arbitrary element of $%
Z^{2}\left( \mathcal{P},\mathcal{P}\right) $. Then $\theta =\left(
0,0,\alpha _{1}\Delta _{3,4},\alpha _{2}\Delta _{3,4}\right) $ for some $%
\alpha _{1},\alpha _{2}\in \mathbb{C}$. The automorphism group of $\mathcal{A%
}_{13}$, $\text{Aut}\left( \mathcal{A}_{13}\right) $, consists of the
automorphisms $\phi $ given by a matrix of the following form:%
\begin{equation*}
\begin{pmatrix}
1 & 0 & 0 & 0 \\ 
0 & 1 & 0 & 0 \\ 
0 & 0 & a_{33} & a_{34} \\ 
0 & 0 & a_{43} & a_{44}%
\end{pmatrix}%
.
\end{equation*}%
Let $\phi =\bigl(a_{ij}\bigr)\in $ $\text{Aut}\left( \mathcal{A}_{13}\right) 
$. Then $\theta \ast \phi =\left( 0,0,\beta _{1}\Delta _{3,4},\beta
_{2}\Delta _{3,4}\right) $ where%
\begin{eqnarray*}
\beta _{1} &=&\alpha _{1}a_{44}-\alpha _{2}a_{34}, \\
\beta _{2} &=&\alpha _{2}a_{33}-\alpha _{1}a_{43}.
\end{eqnarray*}%
$\allowbreak $Then, by Remark \ref{[1 0]}, we may assume $\left( \alpha
_{1},\alpha _{2}\right) \in \left\{ \left( 0,0\right) ,\left( 1,0\right)
\right\} $. So we get the algebras $\mathcal{P}_{4,38}$ and $\mathcal{P}%
_{4,39}$.

\underline{$\left( \mathcal{P},\cdot \right) =\mathcal{A}_{14}$.} Then $%
Z^{2}\left( \mathcal{P},\mathcal{P}\right) =\left\{ 0\right\} $. So we get
the algebra $\mathcal{P}_{4,40}$.

\underline{$\left( \mathcal{P},\cdot \right) =\mathcal{A}_{15}$.} The
automorphism group of $\mathcal{A}_{15}$, $\text{Aut}\left( \mathcal{A}%
_{15}\right) $, consists of the automorphisms $\phi $ given by a matrix of
the following form:%
\begin{equation*}
\begin{pmatrix}
1 & 0 & 0 & 0 \\ 
0 & a_{22} & a_{23} & a_{24} \\ 
0 & a_{32} & a_{33} & a_{34} \\ 
0 & a_{42} & a_{43} & a_{44}%
\end{pmatrix}%
.
\end{equation*}%
Let $\theta =\left( B_{1},B_{2},B_{3},B_{4}\right) $ be an arbitrary element
of $Z^{2}\left( \mathcal{P},\mathcal{P}\right) $. Then%
\begin{eqnarray*}
B_{1} &=&0, \\
B_{2} &=&\alpha _{1}\Delta _{2,3}+\alpha _{2}\Delta _{2,4}+\alpha _{3}\Delta
_{3,4}, \\
B_{3} &=&\alpha _{4}\Delta _{2,3}+\alpha _{5}\Delta _{2,4}+\alpha _{6}\Delta
_{3,4}, \\
B_{4} &=&\alpha _{7}\Delta _{2,3}+\alpha _{8}\Delta _{2,4}+\alpha _{9}\Delta
_{3,4},
\end{eqnarray*}%
such that%
\begin{eqnarray*}
\alpha _{1}\alpha _{8}-\alpha _{2}\alpha _{7}+\alpha _{4}\alpha _{9}-\alpha
_{6}\alpha _{7} &=&0, \\
\alpha _{1}\alpha _{6}+\alpha _{2}\alpha _{9}-\alpha _{3}\alpha _{4}-\alpha
_{3}\alpha _{8} &=&0, \\
\alpha _{1}\alpha _{5}-\alpha _{2}\alpha _{4}-\alpha _{5}\alpha _{9}+\alpha
_{6}\alpha _{8} &=&0,
\end{eqnarray*}%
for some $\alpha _{1},\ldots ,\alpha _{9}\in \mathbb{C}$. Let $\phi =\bigl(%
a_{ij}\bigr)\in $ $\text{Aut}\left( \mathcal{A}_{15}\right) $. Write%
\begin{equation*}
\theta \ast \phi =\left( 0,\beta _{1}\Delta _{2,3}+\beta _{2}\Delta
_{2,4}+\beta _{3}\Delta _{3,4},\beta _{4}\Delta _{2,3}+\beta _{5}\Delta
_{2,4}+\beta _{6}\Delta _{3,4},\beta _{7}\Delta _{2,3}+\beta _{8}\Delta
_{2,4}+\beta _{9}\Delta _{3,4}\right) .
\end{equation*}%
Now if we define $\phi $ to be the following matrix:%
\begin{equation*}
\phi =%
\begin{pmatrix}
1 & 0 & 0 & 0 \\ 
0 & a_{22} & a_{23} & a_{24} \\ 
0 & 0 & a_{33} & a_{34} \\ 
0 & 0 & a_{43} & a_{44}%
\end{pmatrix}%
,
\end{equation*}%
then $%
\begin{pmatrix}
\beta _{4} & \beta _{5} \\ 
\beta _{7} & \beta _{8}%
\end{pmatrix}%
=a_{22}%
\begin{pmatrix}
a_{33} & a_{34} \\ 
a_{43} & a_{44}%
\end{pmatrix}%
^{-1}%
\begin{pmatrix}
\alpha _{4} & \alpha _{5} \\ 
\alpha _{7} & \alpha _{8}%
\end{pmatrix}%
\begin{pmatrix}
a_{33} & a_{34} \\ 
a_{43} & a_{44}%
\end{pmatrix}%
$. So we may assume that%
\begin{equation*}
\begin{pmatrix}
\alpha _{4} & \alpha _{5} \\ 
\alpha _{7} & \alpha _{8}%
\end{pmatrix}%
\in \left\{ 
\begin{pmatrix}
0 & 0 \\ 
0 & 0%
\end{pmatrix}%
,%
\begin{pmatrix}
1 & 0 \\ 
0 & \alpha%
\end{pmatrix}%
,%
\begin{pmatrix}
1 & 1 \\ 
0 & 1%
\end{pmatrix}%
,%
\begin{pmatrix}
0 & 1 \\ 
0 & 0%
\end{pmatrix}%
\right\} .
\end{equation*}%
Then we have the following cases:

\begin{itemize}
\item $%
\begin{pmatrix}
\alpha _{4} & \alpha _{5} \\ 
\alpha _{7} & \alpha _{8}%
\end{pmatrix}%
=%
\begin{pmatrix}
0 & 0 \\ 
0 & 0%
\end{pmatrix}%
$. Again, if we define $\phi $ to be the following matrix:%
\begin{equation*}
\phi =%
\begin{pmatrix}
1 & 0 & 0 & 0 \\ 
0 & a_{22} & a_{23} & a_{24} \\ 
0 & 0 & a_{33} & a_{34} \\ 
0 & 0 & a_{43} & a_{44}%
\end{pmatrix}%
,
\end{equation*}%
then%
\begin{eqnarray*}
\beta _{6} &=&\alpha _{6}a_{44}-\alpha _{9}a_{34}, \\
\beta _{9} &=&\alpha _{9}a_{33}-\alpha _{6}a_{43}.
\end{eqnarray*}%
So, by Remark \ref{[1 0]}, we may assume $\left( \alpha _{6},\alpha
_{9}\right) \in \left\{ \left( 1,0\right) ,\left( 0,0\right) \right\} $.

\begin{itemize}
\item[\ding{118}] $\left( \alpha _{6},\alpha _{9}\right) =\left( 1,0\right) $. Then $%
\alpha _{1}=0$ since otherwise $\theta \notin Z^{2}\left( \mathcal{P},%
\mathcal{P}\right) $. If $\alpha _{2}\neq 1$, we define $\phi $ to be the
following matrix:%
\begin{equation*}
\begin{pmatrix}
1 & 0 & 0 & 0 \\ 
0 & 1 & \frac{\alpha _{3}}{1-\alpha _{2}} & 0 \\ 
0 & 0 & 1 & 0 \\ 
0 & 0 & 0 & 1%
\end{pmatrix}%
.
\end{equation*}%
Then $\theta \ast \phi=\left( 0,\alpha_{2}
\Delta _{2,4},\Delta _{3,4},0\right)$. So we have the representatives $\theta ^{\alpha \neq 1}=\left( 0,\alpha
\Delta _{2,4},\Delta _{3,4},0\right) $. If $\alpha _{2}=1$ and $\alpha
_{3}=0 $, we obtain the representative $\theta ^{\alpha =1}=\left( 0,\Delta
_{2,4},\Delta _{3,4},0\right) $. Morover, the representatives $\theta
^{\alpha },\theta ^{\beta }$ are in the same orbit if and only if $\left(
\alpha -\beta \right) \left( \alpha \beta -1\right) =0$. So we get the
Poisson algebras $\mathcal{P}_{4,41}^{\alpha }$. If $\alpha _{2}=1$ and $%
\alpha _{3}\neq 0$, we define $\phi $ to be the following matrix:%
\begin{equation*}
\begin{pmatrix}
1 & 0 & 0 & 0 \\ 
0 & \alpha _{3} & 0 & 0 \\ 
0 & 0 & 1 & 0 \\ 
0 & 0 & 0 & 1%
\end{pmatrix}%
.
\end{equation*}%
Then $\theta \ast \phi =$ $\left( 0,\Delta _{2,4}+\Delta _{3,4},\Delta
_{3,4},0\right) $. So we get the Poisson algebra $\mathcal{P}_{4,42}$.

\item[\ding{118}] $\left( \alpha _{6},\alpha _{9}\right) =\left( 0,0\right) $. If $%
\alpha _{1}=\alpha _{2}=\alpha _{3}=0$, we get the algebra $\mathcal{P}_{4,43}$%
. Assume now that $\left( \alpha _{1},\alpha _{2},\alpha _{3}\right) \neq
\left( 0,0,0\right) $. If $\left( \alpha _{1},\alpha _{2}\right) \neq \left(
0,0\right) $, we define $\phi $ to be the first of the following matrices if 
$\alpha _{1}\neq 0$ or the second if $\alpha _{1}=0$:%
\begin{equation*}
\phi =%
\begin{pmatrix}
1 & 0 & 0 & 0 \\ 
0 & \frac{1}{\alpha _{1}}\alpha _{3} & \alpha _{1} & 0 \\ 
0 & -\frac{1}{\alpha _{1}}\alpha _{2} & 0 & \frac{1}{\alpha _{1}} \\ 
0 & 1 & 0 & 0%
\end{pmatrix}%
,%
\begin{pmatrix}
1 & 0 & 0 & 0 \\ 
0 & \alpha _{3} & 1 & 0 \\ 
0 & -\alpha _{2} & 0 & 0 \\ 
0 & 0 & 0 & \frac{1}{\alpha _{2}}%
\end{pmatrix}%
.
\end{equation*}%
Then $\theta \ast \phi =\left( 0,0,\Delta _{3,4},0\right) $and we get the
Poisson algebra $\mathcal{P}_{4,41}^{\alpha =0}$. If $\left( \alpha
_{1},\alpha _{2}\right) =\left( 0,0\right) $, we define $\phi $ to be the
following matrix:%
\begin{equation*}
\begin{pmatrix}
1 & 0 & 0 & 0 \\ 
0 & \alpha _{3} & 0 & 0 \\ 
0 & 0 & 1 & 0 \\ 
0 & 0 & 0 & 1%
\end{pmatrix}%
.
\end{equation*}%
Then $\theta \ast \phi =\left( 0,\Delta _{3,4},0,0\right) $. So we get the
Poisson algebra $\mathcal{P}_{4,44}$.
\end{itemize}

\item $%
\begin{pmatrix}
\alpha _{4} & \alpha _{5} \\ 
\alpha _{7} & \alpha _{8}%
\end{pmatrix}%
=%
\begin{pmatrix}
1 & 0 \\ 
0 & \alpha%
\end{pmatrix}%
$. Then
\begin{eqnarray*}
\alpha _{1}\alpha +\alpha _{9} &=&0, \\
\alpha _{1}\alpha _{6}+\alpha _{2}\alpha _{9}-\alpha _{3}-\alpha
_{3}\alpha  &=&0, \\
\alpha _{2}-\alpha
_{6}\alpha &=&0.
\end{eqnarray*}%
Assume first that $\alpha \neq 0$. Set $\lambda =-\frac{1}{\alpha }\left(
\alpha _{2}\alpha _{9}-\alpha\alpha _{3} +\alpha \alpha_{1}\alpha_{6}\right) $. Then $\lambda \left( \alpha +1\right) =0$. If $\lambda =0$,
we choose $\phi $ as follows:%
\begin{equation*}
\begin{pmatrix}
1 & 0 & 0 & 0 \\ 
0 & \alpha _{6} & -\frac{1}{\alpha }\alpha _{9} & -1 \\ 
0 & 0 & 1 & 0 \\ 
0 & 1 & 0 & 0%
\end{pmatrix}%
.
\end{equation*}
Then $\theta \ast \phi =\left( 0,\alpha \Delta _{2,4},\Delta _{3,4},0\right) 
$ . So we obtain the Poisson algebras $\mathcal{P}_{4,41}^{\alpha }$. If $%
\lambda \neq 0$, then $\alpha =-1$. Further if we choose $\phi $ as follows:%
\begin{equation*}
\begin{pmatrix}
1 & 0 & 0 & 0 \\ 
0 & 1 & \alpha _{9} & \frac{1}{\lambda }\alpha _{6} \\ 
0 & 0 & 1 & 0 \\ 
0 & 0 & 0 & \frac{1}{\lambda }%
\end{pmatrix}%
,
\end{equation*}
then $\theta \ast \phi =\left( 0,\Delta _{3,4},\Delta _{2,3},-\Delta
_{2,4}\right) $. Hence we get the Poisson algebras $\mathcal{P}_{4,45}$.
Assume now that $\alpha =0$. Then $\alpha _{2}=\alpha _{9}=0$ and $\alpha
_{3}=\alpha _{1}\alpha _{6}$. Moreover, if we choose $\phi $ as follows:%
\begin{equation*}
\phi =%
\begin{pmatrix}
1 & 0 & 0 & 0 \\ 
0 & \alpha _{6} & \alpha _{1} & -1 \\ 
0 & 0 & 1 & 0 \\ 
0 & 1 & 0 & 0%
\end{pmatrix}%
,
\end{equation*}%
then $\theta \ast \phi =\left( 0,0,\Delta _{3,4},0\right) $. Hence we get
the Poisson algebras $\mathcal{P}_{4,41}^{\alpha =0}$.

\item $%
\begin{pmatrix}
\alpha _{4} & \alpha _{5} \\ 
\alpha _{7} & \alpha _{8}%
\end{pmatrix}%
=%
\begin{pmatrix}
1 & 1 \\ 
0 & 1%
\end{pmatrix}%
$. Then $\alpha_{1}+\alpha_{9}=0,\alpha_{3}+\alpha_{9}^2=0$ and $\alpha_{2}-\alpha_{6}+2\alpha_{9}=0$. Choose $\phi $ as follows:%
\begin{equation*}
\phi =%
\begin{pmatrix}
1 & 0 & 0 & 0 \\ 
0 & -i\alpha _{9} & i\alpha _{6}-\left( 1+i\right) \alpha _{9} & \alpha_{6}-\alpha_{9}-1 \\ 
0 & i & 1 & 0 \\ 
0 & 0 & i & 0%
\end{pmatrix}%
:i=\sqrt{-1}.
\end{equation*}%
Then $\theta \ast \phi =\left( 0,\Delta _{2,4}+\Delta _{3,4},\Delta
_{3,4},0\right) $. So we get the algebra $\mathcal{P}_{4,42}$.

\item $%
\begin{pmatrix}
\alpha _{4} & \alpha _{5} \\ 
\alpha _{7} & \alpha _{8}%
\end{pmatrix}%
=%
\begin{pmatrix}
0 & 1 \\ 
0 & 0%
\end{pmatrix}%
$. Then $\alpha _{1}=\alpha _{9}$ and $\alpha _{9}\left( \alpha _{2}+\alpha
_{6}\right) =0$. \ Let us consider the following cases:

\begin{itemize}
\item[\ding{118}] $\alpha _{9}=\alpha _{2}+\alpha _{6}=0$. Let $\phi $ be the first of
the following matrices if $\alpha _{3}+\alpha _{6}^2\neq 0$ or the
second if $\alpha _{3}+\alpha _{6}^2=0$:%
\begin{equation*}
\begin{pmatrix}
1 & 0 & 0 & 0 \\ 
0 & -\alpha _{6}-\sqrt{\alpha _{3}+\alpha _{6}^2} & \sqrt{\alpha
_{3}+\alpha _{6}^2}-\alpha _{6} & 0 \\ 
0 & 1 & 1 & 0 \\ 
0 & 0 & 0 & \frac{1}{\sqrt{\alpha _{3}+\alpha _{6}^2}}%
\end{pmatrix}%
\allowbreak ,%
\begin{pmatrix}
1 & 0 & 0 & 0 \\ 
0 & -\alpha _{6} & 1 & 0 \\ 
0 & 1 & 0 & 0 \\ 
0 & 0 & 0 & 1%
\end{pmatrix}%
.
\end{equation*}%
Then $\theta \ast \phi =\left( 0,-\Delta _{2,4},\Delta _{3,4},0\right) $ if $%
\alpha _{3}+\alpha _{6}^2\neq 0$ or $\theta \ast \phi =\left(
0,\Delta _{3,4},0,0\right) $ if $\alpha _{3}+\alpha _{6}^2=0$. So
we get the algebras $\mathcal{P}_{41}^{\alpha =-1}$ and $\mathcal{P}_{44}$.

\item[\ding{118}] $\alpha _{9}\neq 0,\alpha _{2}+\alpha _{6}=0$. Let $\phi $ be the
following automorphism: 
\begin{equation*}
\begin{pmatrix}
1 & 0 & 0 & 0 \\ 
0 & \frac{\alpha _{6}}{\alpha _{9}} & -\frac{1}{\alpha _{9}} & \frac{1}{2}%
\frac{\alpha _{6}^{2}}{\alpha _{9}}+\frac{1}{2}\frac{\alpha _{3}}{\alpha _{9}%
} \\ 
0 & -\frac{1}{\alpha _{9}} & 0 & 0 \\ 
0 & 0 & 0 & 1%
\end{pmatrix}%
.
\end{equation*}%
Then $\theta \ast \phi =\left( 0,\Delta _{3,4},\Delta _{2,3},-\Delta
_{2,4}\right) $. So we get the algebra $\mathcal{P}_{4,45}$.

\item[\ding{118}] $\alpha _{9}=0,\alpha _{2}+\alpha _{6}\neq 0$. Consider the following
automorphism:%
\begin{equation*}
\phi =%
\begin{pmatrix}
1 & 0 & 0 & 0 \\ 
0 & \alpha _{2}+\alpha _{6} & -\alpha _{6} & 0 \\ 
0 & 0 & 1 & 0 \\ 
0 & 0 & 0 & \frac{1}{\alpha _{2}+\alpha _{6}}%
\end{pmatrix}%
.
\end{equation*}%
Then $\theta \ast \phi =\left( 0,\Delta _{2,4}+\beta _{3}\Delta
_{3,4},\Delta _{2,4},0\right) $ with $\beta _{3}=\frac{\alpha _{3}-\alpha
_{2}\alpha _{6}}{\left( \alpha _{2}+\alpha _{6}\right) ^{2}}$.

\begin{itemize}
\item[\ding{169}]  $\beta _{3}=0$. Then $\theta \ast \phi \phi ^{\prime }=\left(
0,0,\Delta _{3,4},0\right) $ \ where%
\begin{equation*}
\phi ^{\prime }=%
\begin{pmatrix}
1 & 0 & 0 & 0 \\ 
0 & 0 & 1 & 0 \\ 
0 & 1 & 1 & 0 \\ 
0 & 0 & 0 & 1%
\end{pmatrix}%
.
\end{equation*}%
So we get the algebra $\mathcal{P}_{4,41}^{\alpha =}0$.

\item[\ding{169}] $\beta _{3}=-\frac{1}{4}$. Then $\theta \ast \phi \phi ^{\prime
}=\left( 0,\Delta _{2,4}+\Delta _{3,4},\Delta _{3,4},0\right) $ \ where%
\begin{equation*}
\phi ^{\prime }=%
\begin{pmatrix}
1 & 0 & 0 & 0 \\ 
0 & -\frac{1}{2} & 0 & 0 \\ 
0 & -1 & 1 & 0 \\ 
0 & 0 & 0 & 2%
\end{pmatrix}%
.
\end{equation*}%
So we get the algebra $\mathcal{P}_{4,42}$.

\item[\ding{169}] $\beta _{3}\neq 0,-\frac{1}{4}$. Then $\theta \ast \phi \phi ^{\prime
}=\left( 0,\beta \Delta _{2,4},\Delta _{3,4},0\right) $ for some $\beta \in 
\mathbb{C}
$ where%
\begin{equation*}
\phi ^{\prime }=%
\begin{pmatrix}
1 & 0 & 0 & 0 \\ 
0 & 1 & \frac{1}{2}\left( \frac{\beta _{3}+4\beta _{3}^{2}+\beta _{3}\sqrt{%
1+4\beta _{3}}}{1+4\beta _{3}}\right) & 0 \\ 
0 & \frac{-1-\sqrt{1+4\beta _{3}}}{2\beta _{3}} & \frac{\beta _{3}}{\sqrt{%
1+4\beta _{3}}} & 0 \\ 
0 & 0 & 0 & \frac{1}{2}\left( \frac{-1+\sqrt{1+4\beta _{3}}}{\beta _{3}}%
\right)%
\end{pmatrix}%
.
\end{equation*}%
So we get the algebra $\mathcal{P}_{4,41}^{\alpha =\beta }$.
\end{itemize}
\end{itemize}
\end{itemize}

\underline{$\left( \mathcal{P},\cdot \right) =\mathcal{A}_{16}$.} Let $\theta =\left( B_{1},B_{2},B_{3},B_{4}\right) $ be an arbitrary element
of $Z^{2}\left( \mathcal{P},\mathcal{P}\right) $. Then $\theta =\left(
0,0,\alpha _{1}\Delta _{2,4},\alpha _{2}\Delta _{2,4}\right) $ for some $%
\alpha _{1},\alpha _{2}\in \mathbb{C}$. The automorphism group of $\mathcal{A%
}_{16}$, $\text{Aut}\left( \mathcal{A}_{16}\right) $, consists of the
automorphisms $\phi $ given by a matrix of the following form:%
\begin{equation*}
\begin{pmatrix}
1 & 0 & 0 & 0 \\ 
0 & a_{22} & 0 & 0 \\ 
0 & a_{32} & a_{22}^{2} & a_{34} \\ 
0 & a_{42} & 0 & a_{44}%
\end{pmatrix}%
.
\end{equation*}%
Let $\phi =\bigl(a_{ij}\bigr)\in $ $\text{Aut}\left( \mathcal{A}_{16}\right) 
$. Then $\theta \ast \phi =\left( 0,0,\beta _{1}\Delta _{2,4},\beta
_{2}\Delta _{2,4}\right) $ where%
\begin{eqnarray*}
\beta _{1} &=&\frac{1}{a_{22}}\left( \alpha _{1}a_{44}-\alpha
_{2}a_{34}\right) , \\
\beta _{2} &=&\alpha _{2}a_{22}.
\end{eqnarray*}%
If $\theta =0$, we get the algebra $\mathcal{P}_{4,46}$. Otherwise, let $\phi $\ be the first of
the following matrices if $\alpha _{2}\neq 0$ or the second if $\alpha _{2}=0
$:%
\begin{equation*}
\begin{pmatrix}
1 & 0 & 0 & 0 \\ 
0 & \frac{1}{\alpha _{2}} & 0 & 0 \\ 
0 & 0 & \frac{1}{\alpha _{2}^{2}} & \frac{\alpha _{1}}{\alpha _{2}} \\ 
0 & 0 & 0 & 1%
\end{pmatrix}%
,%
\begin{pmatrix}
1 & 0 & 0 & 0 \\ 
0 & 1 & 0 & 0 \\ 
0 & 0 & 1 & 0 \\ 
0 & 0 & 0 & \frac{1}{\alpha _{1}}%
\end{pmatrix}%
.
\end{equation*}%
Then $\theta \ast \phi =\left( 0,0,0,\Delta _{2,4}\right) $ if $\alpha
_{2}\neq 0$ while $\theta \ast \phi =\left( 0,0,\Delta _{2,4},0\right) $ if $%
\alpha _{2}=0$. So we get the algebras $\mathcal{P}_{4,47}$ and $\mathcal{P}_{4,48}$.

\underline{$\left( \mathcal{P},\cdot \right) =\mathcal{A}_{17}$.} Then $%
Z^{2}\left( \mathcal{P},\mathcal{P}\right) =\left\{ 0\right\} $. So we get
the algebra $\mathcal{P}_{4,49}$.

\underline{$\left( \mathcal{P},\cdot \right) =\mathcal{A}_{18}$.} Let $%
\theta =\left( B_{1},B_{2},B_{3},B_{4}\right) $ be an arbitrary element of $%
Z^{2}\left( \mathcal{P},\mathcal{P}\right) $. Then $\theta =\left(
0,0,0,\alpha \Delta _{2,3}\right) $ for some $\alpha \in \mathbb{C}$. The
automorphism group of $\mathcal{A}_{18}$, $\text{Aut}\left( \mathcal{A}%
_{18}\right) $, consists of the automorphisms $\phi $ given by a matrix of
the following form:%
\begin{equation*}
\begin{pmatrix}
1 & 0 & 0 & 0 \\ 
0 & a_{22} & -\epsilon a_{32} & 0 \\ 
0 & a_{32} & \epsilon a_{22} & 0 \\ 
0 & a_{42} & a_{43} & a_{22}^{2}+a_{32}^{2}%
\end{pmatrix}%
:\epsilon ^{2}=1.
\end{equation*}%
Let $\phi =\bigl(a_{ij}\bigr)\in $ $\text{Aut}\left( \mathcal{A}_{18}\right) 
$. Then $\theta \ast \phi =\left( 0,0,0,\beta \Delta _{2,3}\right) $ where $%
\beta =\epsilon \alpha $. So we get the representatives $\theta ^{\alpha
}=\left( 0,0,0,\alpha \Delta _{2,3}\right) $. Moreover, the representatives $%
\theta ^{\alpha }$ and $\theta ^{\beta }$ are in the same orbit if and only
if $\alpha ^{2}=\beta ^{2}$. Hence we get the algebras $\mathcal{P}%
_{4,50}^{\alpha }$.

\underline{$\left( \mathcal{P},\cdot \right) =\mathcal{A}_{19}$.} Then $%
Z^{2}\left( \mathcal{P},\mathcal{P}\right) =\left\{ 0\right\} $. So we get
the algebra $\mathcal{P}_{4,51}$.

\underline{$\left( \mathcal{P},\cdot \right) =\mathcal{A}_{20}$.} Then $%
Z^{2}\left( \mathcal{P},\mathcal{P}\right) =\left\{ 0\right\} $. So we get
the algebra $\mathcal{P}_{4,52}$.

\underline{$\left( \mathcal{P},\cdot \right) =\mathcal{A}_{21}$.} Let $%
\theta =\left( B_{1},B_{2},B_{3},B_{4}\right) $ be an arbitrary element of $%
Z^{2}\left( \mathcal{P},\mathcal{P}\right) $. Then $\theta =\left( 0,\alpha
_{1}\Delta _{2,3},\alpha _{2}\Delta _{2,3},0\right) $ for some $\alpha
_{1},\alpha _{2}\in \mathbb{C}$. The automorphism group of $\mathcal{A}_{21}$%
, $\text{Aut}\left( \mathcal{A}_{21}\right) $, consists of the automorphisms 
$\phi $ given by a matrix of the following form:%
\begin{equation*}
\begin{pmatrix}
1 & 0 & 0 & 0 \\ 
0 & a_{22} & a_{23} & 0 \\ 
0 & a_{32} & a_{33} & 0 \\ 
0 & 0 & 0 & a_{44}%
\end{pmatrix}%
.
\end{equation*}%
Let $\phi =\bigl(a_{ij}\bigr)\in $ $\text{Aut}\left( \mathcal{A}_{21}\right) 
$. Then $\theta \ast \phi =\left( 0,\beta _{1}\Delta _{2,3},\beta _{2}\Delta
_{2,3},0\right) $ where%
\begin{eqnarray*}
\beta _{1} &=&\alpha _{1}a_{33}-\alpha _{2}a_{23}, \\
\beta _{2} &=&\alpha _{2}a_{22}-\alpha _{1}a_{32}.
\end{eqnarray*}%
$\allowbreak $Then, by Remark \ref{[1 0]}, we may assume $\left( \alpha
_{1},\alpha _{2}\right) \in \left\{ \left( 0,0\right) ,\left( 1,0\right)
\right\} $. So we get the algebras $\mathcal{P}_{4,53}$ and $\mathcal{P}%
_{4,54}$.

\underline{$\left( \mathcal{P},\cdot \right) =\mathcal{A}_{22}$.} Then $%
Z^{2}\left( \mathcal{P},\mathcal{P}\right) =\left\{ 0\right\} $. So we get
the algebra $\mathcal{P}_{4,55}$.

$\allowbreak $\underline{$\left( \mathcal{P},\cdot \right) =\mathcal{A}_{23}$%
.} Let $\theta =\left( B_{1},B_{2},B_{3},B_{4}\right) $ be an arbitrary
element of $Z^{2}\left( \mathcal{P},\mathcal{P}\right) $. Then $\theta
=\left( 0,0,\alpha _{1}\Delta _{3,4},\alpha _{2}\Delta _{3,4}\right) $ for
some $\alpha _{1},\alpha _{2}\in \mathbb{C}$. The automorphism group of $%
\mathcal{A}_{23}$, $\text{Aut}\left( \mathcal{A}_{23}\right) $, consists of
the automorphisms $\phi $ given by a matrix of the following form:%
\begin{equation*}
\begin{pmatrix}
\epsilon & 1-\epsilon & 0 & 0 \\ 
1-\epsilon & \epsilon & 0 & 0 \\ 
0 & 0 & a_{33} & a_{34} \\ 
0 & 0 & a_{43} & a_{44}%
\end{pmatrix}%
:\epsilon \in \left\{ 0,1\right\} .
\end{equation*}%
Let $\phi =\bigl(a_{ij}\bigr)\in $ $\text{Aut}\left( \mathcal{A}_{23}\right) 
$. Then $\theta \ast \phi =\left( 0,0,\beta _{1}\Delta _{3,4},\beta
_{2}\Delta _{3,4}\right) $ where%
\begin{eqnarray*}
\beta _{1} &=&\alpha _{1}a_{44}-\alpha _{2}a_{34}, \\
\beta _{2} &=&\alpha _{2}a_{33}-\alpha _{1}a_{43}.
\end{eqnarray*}%
$\allowbreak $Then, by Remark \ref{[1 0]}, we may assume $\left( \alpha
_{1},\alpha _{2}\right) \in \left\{ \left( 0,0\right) ,\left( 1,0\right)
\right\} $. So we get the algebras $\mathcal{P}_{4,56}$ and $\mathcal{P}%
_{4,57}$.

\underline{$\left( \mathcal{P},\cdot \right) =\mathcal{A}_{24}$.} Then $%
Z^{2}\left( \mathcal{P},\mathcal{P}\right) =\left\{ 0\right\} $. So we get
the algebra $\mathcal{P}_{4,58}$.

\underline{$\left( \mathcal{P},\cdot \right) =\mathcal{A}_{25}$.} Let $%
\theta =\left( B_{1},B_{2},B_{3},B_{4}\right) $ be an arbitrary element of $%
Z^{2}\left( \mathcal{P},\mathcal{P}\right) $. Then $\theta =\left(
0,0,\alpha _{1}\Delta _{3,4},\alpha _{2}\Delta _{3,4}\right) $ for some $%
\alpha _{1},\alpha _{2}\in \mathbb{C}$. The automorphism group of $\mathcal{A%
}_{25}$, $\text{Aut}\left( \mathcal{A}_{25}\right) $, consists of the
automorphisms $\phi $ given by a matrix of the following form:%
\begin{equation*}
\begin{pmatrix}
1 & 0 & 0 & 0 \\ 
0 & a_{22} & 0 & 0 \\ 
0 & 0 & a_{33} & a_{34} \\ 
0 & 0 & a_{43} & a_{44}%
\end{pmatrix}%
.
\end{equation*}%
Let $\phi =\bigl(a_{ij}\bigr)\in $ $\text{Aut}\left( \mathcal{A}_{25}\right) 
$. Then $\theta \ast \phi =\left( 0,0,\beta _{1}\Delta _{3,4},\beta
_{2}\Delta _{3,4}\right) $ where%
\begin{eqnarray*}
\beta _{1} &=&\alpha _{1}a_{44}-\alpha _{2}a_{34}, \\
\beta _{2} &=&\alpha _{2}a_{33}-\alpha _{1}a_{43}.
\end{eqnarray*}%
$\allowbreak $Then, by Remark \ref{[1 0]}, we may assume $\left( \alpha
_{1},\alpha _{2}\right) \in \left\{ \left( 0,0\right) ,\left( 1,0\right)
\right\} $. So we get the algebras $\mathcal{P}_{4,59}$ and $\mathcal{P}%
_{4,60}$.

\underline{$\left( \mathcal{P},\cdot \right) =\mathcal{A}_{26}$.} Then $%
Z^{2}\left( \mathcal{P},\mathcal{P}\right) =\left\{ 0\right\} $. So we get
the algebra $\mathcal{P}_{4,61}$.

\underline{$\left( \mathcal{P},\cdot \right) =\mathcal{A}_{27}$.} The
automorphism group of $\mathcal{A}_{27}$, $\text{Aut}\left( \mathcal{A}%
_{27}\right) $, consists of the automorphisms $\phi $ given by a matrix of
the following form:%
\begin{equation*}
\begin{pmatrix}
1 & 0 & 0 & 0 \\ 
0 & a_{22} & a_{23} & a_{24} \\ 
0 & a_{32} & a_{33} & a_{34} \\ 
0 & a_{42} & a_{43} & a_{44}%
\end{pmatrix}%
.
\end{equation*}%
Let $\theta =\left( B_{1},B_{2},B_{3},B_{4}\right) $ be an arbitrary element
of $Z^{2}\left( \mathcal{P},\mathcal{P}\right) $. Then%
\begin{eqnarray*}
B_{1} &=&0, \\
B_{2} &=&\alpha _{1}\Delta _{2,3}+\alpha _{2}\Delta _{2,4}+\alpha _{3}\Delta
_{3,4}, \\
B_{3} &=&\alpha _{4}\Delta _{2,3}+\alpha _{5}\Delta _{2,4}+\alpha _{6}\Delta
_{3,4}, \\
B_{4} &=&\alpha _{7}\Delta _{2,3}+\alpha _{8}\Delta _{2,4}+\alpha _{9}\Delta
_{3,4},
\end{eqnarray*}%
such that%
\begin{eqnarray*}
\alpha _{1}\alpha _{8}-\alpha _{2}\alpha _{7}+\alpha _{4}\alpha _{9}-\alpha
_{6}\alpha _{7} &=&0, \\
\alpha _{1}\alpha _{6}+\alpha _{2}\alpha _{9}-\alpha _{3}\alpha _{4}-\alpha
_{3}\alpha _{8} &=&0, \\
\alpha _{1}\alpha _{5}-\alpha _{2}\alpha _{4}-\alpha _{5}\alpha _{9}+\alpha
_{6}\alpha _{8} &=&0,
\end{eqnarray*}%
for some $\alpha _{1},\ldots ,\alpha _{9}\in \mathbb{C}$. Since Aut$%
\left( \mathcal{A}_{15}\right) =$Aut$\left( \mathcal{A}_{27}\right) $ and $%
Z^{2}\left( \mathcal{A}_{15},\mathcal{A}_{15}\right) =Z^{2}\left( \mathcal{A}%
_{27},\mathcal{A}_{27}\right) $, we obtain the algebras $\mathcal{P}%
_{4,62}^{\alpha },$ $\mathcal{P}_{4,63},$ $\mathcal{P}_{4,64},\mathcal{P}%
_{4,65},\mathcal{P}_{4,66}$.

\underline{$\left( \mathcal{P},\cdot \right) =\mathcal{A}_{28}$.} Let $%
\theta =\left( B_{1},B_{2},B_{3},B_{4}\right) $ be an arbitrary element of $%
Z^{2}\left( \mathcal{P},\mathcal{P}\right) $. Then $\theta =\left(
0,0,0,\alpha \Delta _{2,3}\right) $ for some $\alpha \in \mathbb{C}$. The
automorphism group of $\mathcal{A}_{28}$, $\text{Aut}\left( \mathcal{A}%
_{28}\right) $, consists of the automorphisms $\phi $ given by a matrix of
the following form:%
\begin{equation*}
\begin{pmatrix}
1 & 0 & 0 & 0 \\ 
0 & \epsilon a_{22} & \left( 1-\epsilon \right) a_{23} & 0 \\ 
0 & \left( 1-\epsilon \right) a_{32} & \epsilon a_{33} & 0 \\ 
0 & a_{42} & a_{43} & \epsilon a_{22}a_{33}+\left( 1-\epsilon \right)
a_{23}a_{32}%
\end{pmatrix}%
:\epsilon \in \left\{ 0,1\right\} .
\end{equation*}%
Let $\phi =\bigl(a_{ij}\bigr)\in $ $\text{Aut}\left( \mathcal{A}_{28}\right) 
$. Then $\theta \ast \phi =\left( 0,0,0,\beta \Delta _{2,3}\right) $ where%
\begin{equation*}
\beta =\frac{\alpha }{a_{23}a_{32}+\epsilon a_{22}a_{33}-\epsilon
a_{23}a_{32}}\left( \epsilon a_{22}a_{33}-a_{23}a_{32}+\epsilon
a_{23}a_{32}\right)
\end{equation*}%
Whence $\beta ^{2}=\alpha ^{2}$. Hence we get the algebras $\mathcal{P}%
_{4,67}^{\alpha }$. Moreover, the algebras $\mathcal{P}_{4,67}^{\alpha }$
and $\mathcal{P}_{4,67}^{\beta }$ are isomorphic if and only if $\alpha
^{2}=\beta ^{2}$.

\underline{$\left( \mathcal{P},\cdot \right) =\mathcal{A}_{29}$.} Let $%
\theta =\left( B_{1},B_{2},B_{3},B_{4}\right) $ be an arbitrary element of $%
Z^{2}\left( \mathcal{P},\mathcal{P}\right) $. Then $\theta =\left(
0,0,\alpha _{1}\Delta _{2,4},\alpha _{2}\Delta _{2,4}\right) $ for some $%
\alpha _{1},\alpha _{2}\in \mathbb{C}$. The automorphism group of $\mathcal{A%
}_{29}$, $\text{Aut}\left( \mathcal{A}_{29}\right) $, consists of the
automorphisms $\phi $ given by a matrix of the following form:%
\begin{equation*}
\begin{pmatrix}
1 & 0 & 0 & 0 \\ 
0 & a_{22} & 0 & 0 \\ 
0 & a_{32} & a_{22}^{2} & a_{34} \\ 
0 & a_{42} & 0 & a_{44}%
\end{pmatrix}%
.
\end{equation*}%
Let $\phi =\bigl(a_{ij}\bigr)\in $ $\text{Aut}\left( \mathcal{A}_{29}\right) 
$. Then $\theta \ast \phi =\left( 0,0,\beta _{1}\Delta _{2,4},\beta
_{2}\Delta _{2,4}\right) $ where%
\begin{eqnarray*}
\beta _{1} &=&\frac{1}{a_{22}}\left( \alpha _{1}a_{44}-\alpha
_{2}a_{34}\right) , \\
\beta _{2} &=&\alpha _{2}a_{22}.
\end{eqnarray*}%
If $\alpha _{2}\neq 0$, we choose $\phi $ as follows:%
\begin{equation*}
\begin{pmatrix}
1 & 0 & 0 & 0 \\ 
0 & \frac{1}{\alpha _{2}} & 0 & 0 \\ 
0 & 0 & \frac{1}{\alpha _{2}^{2}} & \frac{\alpha _{1}}{\alpha _{2}} \\ 
0 & 0 & 0 & 1%
\end{pmatrix}%
.
\end{equation*}%
Then $\theta \ast \phi =\left( 0,0,0,\Delta _{2,4}\right) $. Hence we get
the algebra $\mathcal{P}_{4,68}$. If $\alpha _{2}=0$ and $\alpha _{1}\neq 0$%
, we choose $\phi $ as follows:%
\begin{equation*}
\begin{pmatrix}
1 & 0 & 0 & 0 \\ 
0 & \alpha _{1} & 0 & 0 \\ 
0 & 0 & \alpha _{1}^{2} & 0 \\ 
0 & 0 & 0 & 1%
\end{pmatrix}%
.
\end{equation*}%
Then $\theta \ast \phi =\left( 0,0,\Delta _{2,4},0\right) $. So we get the
algebra $\mathcal{P}_{4,69}$. If $\alpha _{1}=\alpha _{2}=0$, then $\theta
=0 $ and we get the algebra $\mathcal{P}_{4,70}$.

\underline{$\left( \mathcal{P},\cdot \right) =\mathcal{A}_{30}$.} Then $%
Z^{2}\left( \mathcal{P},\mathcal{P}\right) =\left\{ 0\right\} $. So we get
the algebra $\mathcal{P}_{4,71}$.

\bigskip
\bigskip
\bigskip

\noindent {\bf Conflicts of Interest:} The authors declare no conflict of interest.

\bigskip

\noindent {\bf Data Availability Statement:} No new data were created or analyzed in this study. Data sharing is not applicable to this article.

\end{document}